\documentclass[english]{amsart}
\usepackage[T1]{fontenc}
\usepackage[latin9]{inputenc}
\usepackage{geometry}
\usepackage{color}
\usepackage{babel}
\usepackage{mathrsfs}
\usepackage{amsthm}
\usepackage{amstext}
\usepackage{amssymb}
\usepackage{setspace}
\usepackage[unicode=true,pdfusetitle,
 bookmarks=true,bookmarksnumbered=false,bookmarksopen=false,
 breaklinks=false,pdfborder={0 0 1},backref=page,colorlinks=true]
 {hyperref}
\hypersetup{
 colorlinks=true,citecolor=blue,linkcolor=blue,linktocpage=true}

\makeatletter
\numberwithin{equation}{section}
\numberwithin{figure}{section}
\theoremstyle{plain}
\newtheorem{thm}{\protect\theoremname}[section]
  \theoremstyle{plain}
  \newtheorem{prop}[thm]{\protect\propositionname}
  \theoremstyle{plain}
  \newtheorem{cor}[thm]{\protect\corollaryname}
  \theoremstyle{remark}
  \newtheorem{rem}[thm]{\protect\remarkname}
  \theoremstyle{plain}
  \newtheorem{lem}[thm]{\protect\lemmaname}
  \theoremstyle{definition}
  \newtheorem{example}[thm]{\protect\examplename}

\allowdisplaybreaks

\makeatother

  \providecommand{\corollaryname}{Corollary}
  \providecommand{\examplename}{Example}
  \providecommand{\lemmaname}{Lemma}
  \providecommand{\propositionname}{Proposition}
  \providecommand{\remarkname}{Remark}
\providecommand{\theoremname}{Theorem}

\begin{document}
\begin{singlespace}

\title[Intersections of Cantor Sets]{On Intersections of Cantor Sets: Hausdorff Measure}
\end{singlespace}

\author{Steen Pedersen and Jason D. Phillips}

\address{Department of Mathematics, Wright State University, Dayton OH 45435.}

\email{steen@math.wright.edu}

\email{phillips.50@wright.edu}
\begin{abstract}
We establish formulas for bounds on the Haudorff measure of the intersection
of certain Cantor sets with their translates. As a consequence we
obtain a formula for the Hausdorff dimensions of these intersections. 
\end{abstract}
\maketitle
\tableofcontents{}

\section{\label{sec-1:Introduction}Introduction}

Let $n\ge3$ be an integer. Any real number $t\in\left[0,1\right]$
has at least one $n$\emph{-ary} \emph{representation} 
\[
t=0._{n}t_{1}t_{2}\cdots=\sum_{k=1}^{\infty}\frac{t_{k}}{n^{k}}
\]
where each $t_{k}$ is one of the digits $0,1,\ldots,n-1.$ Deleting
some element from the full digit set $\{0,1,\ldots n-1\}$ we get
a set of \emph{digits} $D:=\left\{ d_{k}\mid k=1,2,\ldots,m\right\} $
with $m<n$ digits $d_{k}<d_{k+1}$ and a corresponding \emph{deleted
digits Cantor set} 
\begin{equation}
C=C_{n,D}:=\left\{ \sum_{k=1}^{\infty}\frac{x_{k}}{n^{k}}\mid x_{k}\in D\text{ for all }k\in\mathbb{N}\right\} .\label{Sec-1-eq:C-defined}
\end{equation}
In this paper we investigate the Hausdorff dimension and measure of
the sets $C\cap\left(C+t\right),$ where $C+t:=\{x+t\mid x\in C\}.$
Since the problems we consider are invariant under translation we
will assume $d_{1}=0.$ 

We say that $D$ is \emph{uniform,} if $d_{k+1}-d_{k},$ $k=1,2,\ldots,m-1$
is constant and $\geq2.$ We say $D$ is \emph{regular,} if $D$ is
a subset of a uniform digit set. Finally, we say that $D$ is \emph{sparse,}
if $\left|\delta-\delta'\right|\geq2$ for all $\delta\neq\delta'$
in 
\[
\Delta:=D-D=\left\{ d_{j}-d_{k}\mid d_{j},d_{k}\in D\right\} .
\]
Clearly, a uniform set is regular and a regular set is sparse. The
set $D=\left\{ 0,5,7\right\} $ is sparse and not regular. We will
abuse the terminology and say $C_{n,D}$ is uniform, regular, or sparse
provided $D$ has the corresponding property. 

Previous studies of the sets $C\cap\left(C+t\right)$ include: 
\begin{itemize}
\item When $C=C_{3,\{0,2\}}$ is the middle thirds Cantor set a formula
for the Hausdorff dimension of $C\cap\left(C+t\right)$ can be found
in \cite{DaHu95} and in \cite{NeLi02}. Such a formula can also be
found in \cite{DaTi08} if $C$ is uniform and $d_{m}=n-1$, and in
\cite{KePe91} if $C$ is regular. In Corollary \ref{Sec-1-cor:calculate-dimension}
we establish a formula for the Hausdorff dimension for $C\cap\left(C+t\right)$
when $C$ is sparse. 
\item Let $F^{+}$ be the set of all $t\geq0$ such that $C\cap\left(C+t\right)$
is non-empty. For $0\leq\beta\leq1,$ let $F_{\beta}:=\left\{ t\in F^{+}\mid\dim\left(C\cap\left(C+t\right)\right)=\beta\log_{n}(m)\right\} ,$
where $\dim\left(C\cap\left(C+t\right)\right)$ is the Hausdorff dimension
of $C\cap\left(C+t\right).$ If $C$ is the middle thirds Cantor set
then $F^{+}=[0,1]$ and it is shown in \cite{Ha75,DaHu95,NeLi02}
that $F_{\beta}$ is dense in $F^{+}$ for all $0\leq\beta\leq1.$
This extended to regular set and to sets $C_{n,D}$ such that $D$
satisfies $d_{k+1}-d_{k}\geq2$ and $d_{m}<n-1$ in \cite{PePh11}.
It is also shown in \cite{PePh11} that $F_{\beta}$ is not dense
in $F^{+}$ for all $0\leq\beta\leq1$ for all deleted digits Cantor
sets $C_{n,D}.$ We address this problem for the Hausdorff measure
in place of the Hausdorff dimension when $D$ is sparse in Corollary
\ref{cor:F_b,A dense in F}. 
\item It is shown in \cite{Ha75,Ig03} that, if $C$ is the middle thirds
Cantor set, then the Hausdorff dimension of $C\cap\left(C+t\right)$
is $\frac{1}{3}\log_{3}(2)$ for Lebesgue almost all $t$ in the closed
interval $\left[0,1\right].$ This is is extended to all deleted digits
sets in \cite{KePe91}.
\item If $C$ is the middle thirds Cantor set, then $C\cap\left(C+t\right)$
is self-similar if and only if the sequence $\left\{ 1-\left|y_{k}\right|\right\} $
is strong periodic where $t=\sum_{k=1}^{\infty}\frac{2y_{k}}{3^{k}}$
and $y_{k}\in\left\{ -1,0,1\right\} $ for all $k$ by \cite{LYZ11}.
Thus, $C\cap\left(C+t\right)$ is not, in general, a self-similar
set. 
\item For the middle thirds Cantor set it is shown in \cite{DaHu95,NeLi02}
that $C\cap\left(C+t\right)$ has $\log_{3}(2)$--dimensional Hausdorff
measure $0$ or $\frac{1}{2^{k}}$ for some integer $k.$ This is
extended to $\log_{n}(m)$--dimensional Hausdorff measure for uniform
sets with $d_{m}=n-1$ in \cite{DaTi08}. In Theorem \ref{Sec-1-thm:estimate-measure}
we estimate the $s-$dimensional Hausdorff measure of $C\cap\left(C+t\right),$
when $D$ is sparse and $s$ is the Hausdorff dimension of $C\cap\left(C+t\right).$ 
\end{itemize}
Some of the cited papers only consider rational $t$ and some consider
Minkowski dimension in place of Hausdorff dimension. It is known,
see e.g., \cite{PePh11} for an elementary proof, that the (lower)
Minkowski dimensions of $C\cap\left(C+t\right)$ equals its Hausdorff
dimension. 

Palis \cite{Pa87} conjectured that for dynamically defined Cantor
sets typically the corresponding set $F^{+}$ either has Lebesgue
measure zero or contains an interval. The papers \cite{DeSi08,MSS09}
investigate this problem for random deleted digits sets and solve
it in the affirmative in the deterministic case. 

For $n$-ary representations $t=0._{n}t_{1}t_{2}\ldots$ with $t_{k}\in\{0,1,\ldots,n-1\},$
let $\left\lfloor t\right\rfloor _{k}:=\sum_{j=1}^{k}\frac{t_{j}}{n^{j}}=0._{n}t_{1}t_{2}\cdots t_{k}$
denote the \emph{truncation} of $t$ to the first $k$ $n$-ary places.
Note that the truncation of $t$ is unique, unless $t$ admits two
different $n$-ary representations.

The case where $t$ admits a finite $n$-ary representation is relatively
simple. In fact, Theorem \ref{Sec-3-thm:Finite-Representations} shows
that, if $t=0._{n}t_{1}t_{2}\cdots t_{k},$ then $C\cap\left(C+t\right)$
is a union of two, possibly empty, sets $A$ and $B,$ where $A$
is a finite disjoint union of sets of the form $\frac{1}{n^{k}}\left(C+h\right)$
and $B$ is a finite set. Consequently, we will focus on translations
$t$ that do not admit a finite $n$-ary representation. 

Let
\[
C_{k}:=\left\{ 0._{n}x_{1}x_{2}\ldots\mid x_{j}\in D\text{ for }1\le j\le k\right\} 
\]
for each $k,$ then $C_{0}=[0,1],$ 
\begin{equation}
C_{k+1}\subset C_{k},\text{ and }C=C_{n,D}=\bigcap_{k=0}^{\infty}C_{k}.\label{Sec2-eq:C-construction-by-intersections}
\end{equation}

Let $0\leq t\leq1$ be fixed. Let $J=\frac{1}{n^{k}}\left(C_{0}+h\right)$
be an interval contained in $C_{k}$ for some integer $h.$ We say
$J$ is in the \emph{interval case}, if it is also an interval in
$C_{k}+\left\lfloor t\right\rfloor _{k}.$ And we say $J$ is in the
\emph{potential interval} case, if $J+\frac{1}{n^{k}}$ is an interval
in $C_{k}+\left\lfloor t\right\rfloor _{k}.$ 
\begin{prop}
\label{Sec-1-prop:Either-or} Suppose $D$ is sparse. Let $0\leq t\leq1.$
If one of the intervals in $C_{k}$ is in the interval case, then
no interval in $C_{k}$ is in the potential interval case. If one
of the intervals in $C_{k}$ is in the potential interval case, then
no interval in $C_{k}$ is in the interval case. 
\end{prop}
Suppose $D$ is sparse and $t=0._{n}t_{1}t_{2}\cdots.$ Let $\mu_{t}(0)=1$
and inductively $\mu_{t}(k+1)=\mu_{t}(k)\cdot\#\left(D-t_{k+1}\right)\cap\left(D\cup\left(D+1\right)\right)$
if one of the intervals in $C_{k}$ is in the interval case, $\mu_{t}(k+1)=\mu_{t}(k)\cdot\#\left(D-n+t_{k+1}\right)\cap\left(D\cup\left(D-1\right)\right)$
if one of the intervals in $C_{k}$ is in the potential interval case,
and $\mu_{t}(k+1)=0$ if no interval in $C_{k}$ is in the interval
or potential interval case. Here $\#B$ denotes the number of elements
in the finite set $B.$ Let $\nu_{t}(k):=\log_{m}\mu_{t}\left(k\right),$
$\beta_{t}:=\liminf_{k\to\infty}\frac{\nu_{t}(k)}{k},$ and $L_{t}:=\liminf_{k\to\infty}m^{\nu_{t}(k)-k\beta_{t}}.$
These numbers all depend on $n$ and $D,$ but we suppress this dependence
in the notation. A special case of Theorem \ref{thm:m^-bL<H<L} is 
\begin{thm}
\label{Sec-1-thm:estimate-measure}Let $C=C_{n,D}$ be a deleted digits
Cantor set. Suppose $D$ is sparse, $0<t<1$ does not admit a finite
$n$-ary representation, and $C\cap\left(C+t\right)$ is non-empty.
If $s:=\beta_{t}\log_{n}(m),$ then 
\[
m^{-\beta_{t}}L_{t}\leq\mathscr{H}^{s}\left(C\cap\left(C+t\right)\right)\leq L_{t},
\]
where $\mathscr{H}^{s}\left(C\cap\left(C+t\right)\right)$ is the
$s$-dimensional Hausdorff measure of $C\cap\left(C+t\right).$ 
\end{thm}
\textcolor{black}{We also show, see Remark \ref{Sec-5:Remark-smaller-upper-bound},
that Lemma \ref{Sec-4-Lem:Interval_Count} leads to a smaller upper
bound at the expense of a more complicated expression for this upper
bound. We also present an example, Example \ref{ex:< M_t =00003D L_t},
showing that this smaller upper bound need not be equal to the Hausdorff
measure of $C\cap\left(C+t\right).$}
\begin{cor}
\label{Sec-1-cor:calculate-dimension}Let $C=C_{n,D}$ be a deleted
digits Cantor set. If $D$ is sparse, $0<t<1$ does not admit a finite
$n$-ary representation, and $C\cap\left(C+t\right)$ is non-empty,
then $C\cap\left(C+t\right)$ has Hausdorff dimension $\beta_{t}\log_{n}(m).$ 
\end{cor}
As noted above the sets $C\cap\left(C+t\right)$ are usually not self-similar.
In Example \ref{ex:L_s=00003D0 and L_r=00003Dinfty} we construct
$C$ and $t$ such that $C\cap\left(C+t\right)$ has Hausdorff dimension
$\beta\log_{n}(m)$ and $L=0$ or $L=\infty.$ In these cases $C\cap\left(C+t\right)$
is not self-similar and Theorem \ref{Sec-1-thm:estimate-measure}
provides a formula for the Hausdorff measure. We show, Theorem \ref{thm:Finite Measure},
that our proof of Theorem \ref{Sec-1-thm:estimate-measure} can be
modified to give the estimate $m^{-1}\leq\mathscr{H}^{s}\left(C\right)\leq1,$
where $s=\log_{n}(m).$ A formula for the Hausdorff measure of self-similar
sets is not known except in very special circumstances. However, the
papers \cite{AySt99,Ma86,Ma87} contain algorithms for calculating
the Hausdorff measure of self-similar subsets of the real line satisfying
an open set condition. Corollary \ref{Sec-5-cor:estimate-for-finite-t}
contain estimates on the Hausdorff measure of $C\cap\left(C+t\right)$
when $t$ admits a finite $n$-ary representation. 

In Section \ref{sec:Examples} we give examples showing that $\mathscr{H}^{s}\left(C\cap\left(C+t\right)\right)$
can but need not equal $L_{t}.$ We also present an example showing
that if $D$ is not sparse, then $\mathscr{H}^{s}\left(C\cap\left(C+t\right)\right)$
need not be in the interval $\left[m^{-\beta_{t}}L_{t},L_{t}\right].$ 

We refer to \cite{Fal85} for background information on Hausdorff
dimension, Hausdorff measure and self-similar sets. This paper is
based, in part, on the second named authors' thesis \cite{Phi11}.

\textcolor{black}{After this work was completed, we became aware of
earlier works \cite{FWW97} \cite{QRS01}, on these problems. These
papers consider a class of Cantor sets similar to but larger than
uniform deleted digits sets with $d_{m}=n.$ We refer to this class
as }\textcolor{black}{\emph{homogeneous Cantor sets}}\textcolor{black}{{}
and refer to the cited papers for the exact definition. The first
of these papers, \cite{FWW97}, establishes an estimate for homogeneous
Cantor sets, similar to our Theorem \ref{Sec-1-thm:estimate-measure}
with $t=0.$ The second of these papers, \cite{QRS01}, shows that,
for a smaller class of homogeneous Cantor sets, the upper bound in
\cite{FWW97} is in fact equal to the Hausdorff measure. It is likely
that these results, combined with the analysis in Sections \ref{sec-2}--\ref{sec-4},
can be used to establish a formula for the Hausdorff measure of $C\cap\left(C+t\right),$
when $C$ is a uniform deleted digits Cantor set and $d_{m}=n.$}

\section{A Construction of $C\cap(C+t)$\label{sec-2}}

In this section we assume $n\geq3$ is given and that $D=\left\{ d_{k}\mid k=1,2,\ldots\right\} $
is some digits set. We indicate how a natural method of construction
of $C$ can be used to analyze $C\cap\left(C+t\right)$. This contruction
form the basis for our analysis of $C\cap\left(C+t\right).$ 

The middle thirds Cantor set is often constructed by starting with
the closed interval $C_{0}=[0,1]$ and for each $k\geq0$ letting
$C_{k+1}$ be obtained from $C_{k}$ be removing the open middle of
each interval in $C_{k}.$ We show that $C=C_{n,D}$ can be constructed
in a similar manner. 

The \emph{refinement} of the interval $[a,b]$ is the set
\[
\bigcup_{j=1}^{m}\left[a+\frac{d_{j}}{n}\left(b-a\right),a+\frac{d_{j}+1}{n}\left(b-a\right)\right].
\]
The set $C_{k+1}$ is obtained from $C_{k}$ by refining each $n$-ary
interval in $C_{k}.$ For the middle thirds Cantor set refinement
of $C_{k}$ is the same as removing the open middle third from each
interval in $C_{k}.$ 

Since we are interested in studying $C\cap\left(C+t\right)$ only
$t$ such that $C\cap\left(C+t\right)$ is not empty are of interest.
Consequently we introduce the set 
\[
F:=\left\{ t\mid C\cap\left(C+t\right)\neq\varnothing\right\} .
\]

It is easy to see that $F$ is compact and $F=C-C$. As a result,
$F^{+}=F\cap[0,\infty)$ and $F=\left(-F\right)\cup F.$ Since $C\cap\left(C-t\right)$
is translate of $C\cap\left(C+t\right)$ it is sufficient to consider
$t\geq0.$ 
\begin{rem}
It is shown in \cite{PePh11} that $F$ is the compact set $\left\{ 0._{n}t_{1}t_{2}\cdots\mid t_{k}\in\Delta\right\} $.
In particular, $F$ is a self-similar set. Note the representations
$0._{n}t_{1}t_{2}\cdots$ with $t_{k}\in\Delta$ allows the digits
$t_{k}$ to be positive for some $k$ and negative for other $k.$
We will not need this construction of $F$ in this paper. 
\end{rem}
Fix $t=0._{n}t_{1}t_{2}\ldots$ in $[0,1].$ We split our analysis
of $C\cap\left(C+t\right)$ into three steps. First, we consider the
method of construction for the sets $C_{k}\cap\left(C_{k}+\left\lfloor t\right\rfloor _{k}\right)$.
Second, we establish a relationship between $C_{k}\cap\left(C_{k}+\left\lfloor t\right\rfloor _{k}\right)$
and $C_{k}\cap\left(C_{k}+t\right).$ Thirdly, this allows us to use
that (\ref{Sec2-eq:C-construction-by-intersections}) implies 
\begin{equation}
C\cap\left(C+t\right)=\bigcap_{k=0}^{\infty}\left(C_{k}\cap\left(C_{k}+t\right)\right)\label{Sec-2-eq:C_t-as-intersection-of- C_ks}
\end{equation}
to investigate $C\cap\left(C+t\right).$

\section{Analysis of $C_{k}\cap\left(C_{k}+\left\lfloor t\right\rfloor _{k}\right)$\label{sec-3}}

Given any $h\in\mathbb{Z}$ we say that the interval $J=\frac{1}{n^{k}}\left(C_{0}+h\right)$
is an \emph{$n$-ary interval} of length $\frac{1}{n^{k}}$\emph{.}
We will simply say $n$-ary interval when $k$ is understood from
the context. In particular, if $U$ is a compact set, the phrase \emph{an
$n$-ary interval of $U$} refers to an $n$-ary interval of length
$\frac{1}{n^{k}}$ contained in $U$ where $k$ is the smallest such
$k$. In particular, $C_{k}$ consists of $m^{k}$ disjoint $n$-ary
intervals. 

Fix $t=0._{n}t_{1}t_{2}\ldots$ in $[0,1].$ To construct $C_{k}\cap\left(C_{k}+\left\lfloor t\right\rfloor _{k}\right)$
we begin by generating $C_{k+1}$ by refining each $n$-ary interval
of $C_{k}$. Note that $\left\lfloor t\right\rfloor _{k}=\frac{h}{n^{k}}$
for some positive integer $h$ so that $C_{k}+\left\lfloor t\right\rfloor _{k}$
also consists of $n$-ary intervals. Thus, $C_{k+1}+\left\lfloor t\right\rfloor _{k+1}$
is generated by first refining each $n$-ary interval of $C_{k}+\left\lfloor t\right\rfloor _{k}$
and then translating these refined intervals by the positive factor
$\frac{t_{k+1}}{n^{k+1}}$. We say that $C_{k}\cap\left(C_{k}+\left\lfloor t\right\rfloor _{k}\right)$
\emph{transitions to} $C_{k+1}\cap\left(C_{k+1}+\left\lfloor t\right\rfloor _{k+1}\right)$
by first generating the sets $C_{k+1}$ and $C_{k+1}+\left\lfloor t\right\rfloor _{k+1}$
and then taking their intersection. 

Let $J\subset C_{k}$ be an arbitrary $n$-ary interval. Then $J$
can be classified using combinations of the following four cases:
(1) $J$ also in an $n$-ary interval in $C_{k}+\left\lfloor t\right\rfloor _{k},$
(2) the left hand end point of $J$ is the right hand end point of
some $n$-ary interval in $C_{k}+\left\lfloor t\right\rfloor _{k},$
(3) the right hand end point of $J$ is the left hand end point of
some $n$-ary interval in $C_{k}+\left\lfloor t\right\rfloor _{k},$
or (4) $J$ does not have any points in common with $C_{k}+\left\lfloor t\right\rfloor _{k}.$
More specifically, let $J$ be an $n$-ary interval in $C_{k}.$ 
\begin{enumerate}
\item We say $J$ is in the \emph{interval case,} if there exists an $n$-ary
interval $K\subset C_{k}+\left\lfloor t\right\rfloor _{k}$ such that
$J=K.$
\item We say $J$ is in the \emph{potential interval case,} if there exists
an $n$-ary interval $K\subset C_{k}+\left\lfloor t\right\rfloor _{k}$
such that $J=K+\frac{1}{n^{k}}$. 
\item We say $J$ is in the \emph{potentially empty case,} if there exists
an $n$-ary interval $K\subset C_{k}+\left\lfloor t\right\rfloor _{k}$
such that $J=K-\frac{1}{n^{k}}$. 
\item We say $J$ is in the \emph{empty case,} if $J\cap\left(C_{k}+\left\lfloor t\right\rfloor _{k}\right)=\varnothing$. 
\end{enumerate}
Any $n$-ary interval in $C_{k}$ is in one or more of the four cases
described above. An $n$-ary interval $J$ in $C_{k}$ may both in
the interval case and in the potential interval case, i.e., there
exists $n$-ary intervals $K_{I},K_{P}\subset\left(C_{k}+\left\lfloor t\right\rfloor _{k}\right)$
such that $K_{P}+\frac{1}{n^{k}}=J=K_{I}$. It is also possible for
$J$ to be in both the interval case and potentially empty case, or
to be both in the potential interval case and in the potentially empty
case. However, the intersections corresponding to the potentially
empty cases do not contribute points to $C\cap\left(C+t\right),$
when $0<t-\left\lfloor t\right\rfloor _{k}.$ Hence, we will not identify
these cases with special terminology. Finally, any $J$ in the empty
case cannot also be in any of the other cases. 

The idea of our method is to take $n$-ary interval in $C_{k}$ and
use the above classification to investigate the intersection $J\cap C\cap\left(C+t\right).$
The basic question is whether or not this intersection is non-empty?
whether or not repeated refinement of $J$ ''leads to'' points in
$C\cap\left(C+t\right)?$

\subsection{Finite $n$-ary Representations.}

We show that, if $t\in F^{+}$ admits a finite $n$-ary representation,
then $C\cap\left(C+t\right)$ is a union of finite sets and sets similar
to $C$. 
\begin{thm}
\label{Sec-3-thm:Finite-Representations}Suppose $t=0._{n}t_{1}t_{2}\cdots t_{k}$
is in $F^{+}.$ Then
\[
C\cap\left(C+t\right)=A\cup B
\]
where $A$ is empty or $A=\bigcup_{j}\frac{1}{n^{k}}\left(C+h_{j}\right)$
for a finite set of integers $h_{j}$ and $B$ is a finite, perhaps
empty, set. More precisely, each $n$-ary interval in $C_{k}$ that
is in the interval case gives rise to a term in the union in $A.$
If $d_{m}<n-1,$ then $B$ is empty. If $d_{m}=n-1,$ then (i) each
$n$-ary interval in $C_{k}$ that is in the potential interval case
and not in the potentially empty case gives rise to one point in $B$
(ii) each $n$-ary interval in $C_{k}$ that is in the potentially
empty case and not in the potential interval case gives rise to one
point in $B,$ and (iii) each $n$-ary interval in $C_{k}$ that both
is in the potential interval case and in the potentially empty case
gives rise to two point in $B.$ \end{thm}
\begin{proof}
Let $J_{0}$ be an $n$-ary interval in $C_{k}$ and let $h$ be the
integer for which $J_{0}=\frac{1}{n^{k}}\left(C_{0}+h\right).$ 

Suppose $J_{0}$ is in the interval case. For $j\geq0$ let $J_{j+1}$
be obtained from $J_{j}$ by refining each interval in $J_{j}.$ Since
$C_{\ell+1}$ is obtained from $C_{\ell}$ by refining each interval
in $C_{\ell},$ it follows that $J_{j}=\frac{1}{n^{k}}\left(C_{j}+h\right)$
for all $j\geq0.$ So (\ref{Sec2-eq:C-construction-by-intersections})
implies 
\begin{equation}
\bigcap_{j=0}^{\infty}J_{j}=\frac{1}{n^{k}}\left(C+h\right).\label{sec-3-eq:J-intersection}
\end{equation}

Consider the transition from $C_{k}\cap\left(C_{k}+\left\lfloor t\right\rfloor _{k}\right)$
to $C_{k+1}\cap\left(C_{k+1}+\left\lfloor t\right\rfloor _{k+1}\right).$
By assumption $J_{0}\subseteq C_{k}$ and $J_{0}\subseteq C_{k}+\left\lfloor t\right\rfloor _{k}.$
Applying the refinement process to all intervals gives $J_{1}\subseteq C_{k+1}$
and $J_{1}\subseteq C_{k+1}+\left\lfloor t\right\rfloor _{k}.$ Since
$\left\lfloor t\right\rfloor _{k}=\left\lfloor t\right\rfloor _{k+1}$
we conclude $J_{1}\subseteq C_{k+1}\cap\left(C_{k+1}+\left\lfloor t\right\rfloor _{k+1}\right)=C_{k+1}\cap\left(C_{k+1}+t\right).$
Repeating this argument shows that $J_{j}\subseteq C_{k+j}\cap\left(C_{k+j}+t\right)$
for all $j\geq0.$ Hence combining (\ref{sec-3-eq:J-intersection})
and (\ref{Sec-2-eq:C_t-as-intersection-of- C_ks}) we conclude 
\[
\frac{1}{n^{k}}\left(C+h\right)\subseteq C\cap\left(C+t\right).
\]
Thus any interval in $C_{k}$ that is in the interval case gives rise
to a ``small copy'' of $C$ in $C\cap\left(C+t\right).$ 

Suppose $J_{0}$ is in the potential interval case. Then $K_{0}:=J_{0}-\frac{1}{n^{k}}$
is an $n$-ary interval in $C_{k}+\left\lfloor t\right\rfloor _{k}.$
The refinements of $J_{0}$ and $K_{0}$ are 
\[
J_{1}=\bigcup_{p=1}^{m}\frac{1}{n^{k+1}}\left(C_{0}+nh+d_{p}\right)\text{ and }K_{1}=\bigcup_{p=1}^{m}\frac{1}{n^{k+1}}\left(C_{0}+nh+d_{p}-n\right).
\]
Since $C_{0}$ is a closed interval of length one, $0\leq d_{q}\leq d_{q+1}\leq n-1$,
$J_{1}\cap K_{1}$ is non-empty iff $d_{0}=1+d_{m}-n$ iff $d_{m}=n-1.$
In the affirmative case $J_{0}\cap K_{0}=J_{1}\cap K_{1}.$ Since
$t=\left\lfloor t\right\rfloor _{k}=\left\lfloor t\right\rfloor _{k+1}$
we have 
\[
C\cap\left(C+t\right)\supseteq\left(J_{0}\cap C\right)\cap\left(K_{0}\cap\left(C+t\right)\right)\supseteq\left(J_{1}\cap C\right)\cap\left(K_{1}\cap\left(C+t\right)\right).
\]
Hence, $J_{0}\cap K_{0}$ is a point in $C\cap\left(C+t\right)$ iff
$d_{m}=n-1$. 

The case where $J_{0}$ is in the potentially empty case is similar
to the case where $J_{0}$ is in the potential intervals case. 

Finally, suppose $J_{0}$ is in the empty case. Since $t=\left\lfloor t\right\rfloor _{k}$
it follows from (\ref{Sec2-eq:C-construction-by-intersections}) that
$J_{0}\cap\left(C+t\right)\subseteq J_{0}\cap\left(C_{k}+\left\lfloor t\right\rfloor _{k}\right).$
But the right hand side is the empty set by assumption. \end{proof}
\begin{rem}
The sets $A$ and $B$ in Theorem \ref{Sec-3-thm:Finite-Representations}
need not be disjoint. 
\end{rem}

\subsection{Infinite $n$-ary Representations}

Theorem \ref{Sec-3-thm:Finite-Representations} provides us with complete
information about $C\cap\left(C+t\right),$ when $t$ admits a finite
$n$-ary representation. Consequently, it remains to investigate $C\cap\left(C+t\right)$
when $t$ does not admit such a finite representation, i.e., when
\begin{equation}
0<t-\left\lfloor t\right\rfloor _{k}<\frac{1}{n^{k}}\text{ for all }k\geq1.\label{Sec-3-eq:t-non-teminating}
\end{equation}
Our next result shows that, if $t$ does not admit a finite $n$-ary
representation, then only interval and potential interval cases can
contribute points to $C\cap\left(C+t\right)$. 
\begin{lem}
\label{Sec-3-Lem:empty-can-be-ignored}Suppose $0<t-\left\lfloor t\right\rfloor _{k}<\frac{1}{n^{k}}$
for some $k$. If $J$ is an $n$-ary interval in $C_{k}$ and $J$
is either in the potentially empty or the empty case, then $J\cap\left(C_{k}+t\right)$
is empty, in particular, the intersection $J\cap C\cap\left(C+t\right)$
is empty. \end{lem}
\begin{proof}
Suppose $t=0._{n}t_{1}t_{2}\cdots$ satisfies $0<t-\left\lfloor t\right\rfloor _{k}<\frac{1}{n^{k}}$
for some $k$. Let $J$ be an $n$-ary interval in $C_{k}$ and let
$K$ be an $n$-ary interval in $C_{k}+\left\lfloor t\right\rfloor _{k}$.
Pick integers $h_{J}$ and $h_{K}$ such that $J=\frac{1}{n^{k}}\left(C_{0}+h_{J}\right)$
and $K=\frac{1}{n^{k}}\left(C_{0}+h_{K}\right)$ 

Suppose $J$ is in the potentially empty case and $K$ is such that
$J=K-\frac{1}{n^{k}}$. Then $h_{J}=h_{K}-1.$ Hence, $0<t-\left\lfloor t\right\rfloor _{k}$
implies
\[
J\cap\left(K+\left(t-\left\lfloor t\right\rfloor _{k}\right)\right)=\frac{1}{n^{k}}\left(\left(C_{0}+h_{J}\right)\cap\left(C_{0}+h_{J}+1+\left(t-\left\lfloor t\right\rfloor _{k}\right)n^{k}\right)\right)=\varnothing,
\]
 since $C_{0}$ is an interval of length one. By (\ref{Sec-2-eq:C_t-as-intersection-of- C_ks})
this intersection does not contribute any points to $C\cap\left(C+t\right)$.

Suppose $J$ is in the empty case. Let $K\subset C_{k}+\left\lfloor t\right\rfloor _{k}$
be an arbitrary $n$-ary interval. Since $J$ is a minimum distance
of $\frac{1}{n^{k}}$ from $K$, then $K+\left(t-\left\lfloor t\right\rfloor _{k}\right)$
is at least a distance $\frac{1}{n^{k}}-\left(t-\left\lfloor t\right\rfloor _{k}\right)>0$
from $J$. Hence, $J\cap\left(C_{k}+t\right)=\varnothing$ and $J$
does not contain any points of $C\cap\left(C+t\right)$. \end{proof}
\begin{rem}
\label{Sec-3-Remark:lengts-of-intersections} The arguments from the
proof of Lemma \ref{Sec-3-Lem:empty-can-be-ignored} also give some
information about the interval and potential interval cases when $t$
does not admit a finite $n$-ary representation. More precisely, suppose
$J$ is in the interval case and $K\subset C_{k}+\left\lfloor t\right\rfloor _{k}$
is an $n$-ary interval such that $K=J$. Since $t-\left\lfloor t\right\rfloor _{k}<\frac{1}{n^{k}}$
then $J\cap\left(K+\left(t-\left\lfloor t\right\rfloor _{k}\right)\right)$
is an interval of length $\frac{1}{n^{k}}-\left(t-\left\lfloor t\right\rfloor _{k}\right)>0$
contained in $C_{k}\cap\left(C_{k}+t\right)$ which therefore may
contain points of $C\cap\left(C+t\right)$.

Suppose $J$ is in the potential interval case and $K$ is an $n$-ary
interval in $C_{k}+\left\lfloor t\right\rfloor _{k}$ such that $K+\frac{1}{n^{k}}=J$.
Since $0<\left(t-\left\lfloor t\right\rfloor _{k}\right),$ then $J\cap\left(K+t-\left\lfloor t\right\rfloor _{k}\right)$
is an interval of length $t-\left\lfloor t\right\rfloor _{k}$ and
this intersection may therefore contain points of $C\cap\left(C+t\right)$.
\end{rem}

\section{Analysis of the transition from $C_{k}\cap\left(C_{k}+\left\lfloor t\right\rfloor _{k}\right)$
to $C_{k+1}\cap\left(C_{k+1}+\left\lfloor t\right\rfloor _{k+1}\right)$\label{sec-4}}

Fix $t=0._{n}t_{1}t_{2}\ldots$ in $[0,1].$ We begin by considering
what happens to an $n$-ary interval $J$ in $C_{k}$ that is in the
interval case or the potential intervals case when we transition from
$C_{k}\cap\left(C_{k}+\left\lfloor t\right\rfloor _{k}\right)$ to
$C_{k+1}\cap\left(C_{k+1}+\left\lfloor t\right\rfloor _{k+1}\right).$
\begin{lem}
\label{Sec-4-Lemma:Transition} Let $J\subset C_{k}$ and $K\subset C_{k}+\left\lfloor t\right\rfloor _{k}$
be $n$-ary intervals and let $t=0._{n}t_{1}t_{2}\ldots$ be some
point in $[0,1].$ 
\begin{enumerate}
\item Suppose $J=K.$ (Interval case)

\begin{enumerate}
\item If $t_{k+1}$ is in $\Delta,$ then exactly $\#D\cap\left(D+t_{k+1}\right)$
of the intervals in the refinement of $J$ are in the interval case.
\item If $t_{k+1}$ is in $\Delta-1,$ then exactly $\#D\cap\left(D+t_{k+1}+1\right)$
of the intervals the refinement of $J$ are in the potential interval
case. 
\item If $t_{k+1}$ is neither in $\Delta$ nor in $\Delta-1,$ then all
intervals in the refinement of $J$ are either in the empty case or
in the potentially empty case. 
\end{enumerate}
\item Suppose $J=K+\frac{1}{n^{k}}.$ (Potential interval case)

\begin{enumerate}
\item If $t_{k+1}$ is in $n-\Delta,$ then exactly $\#D\cap\left(D+n-t_{k+1}\right)$
of the intervals in the refinement of $J$ are in the interval case.
\item If $t_{k+1}$ is in $n-\Delta-1,$ then exactly $\#D\cap\left(D+n-t_{k+1}-1\right)$
of the intervals the refinement of $J$ are in the potential interval
case. 
\item If $t_{k+1}$ is neither in $n-\Delta$ nor in $n-\Delta-1,$ then
all intervals in the refinement of $J$ are either in the empty case
or in the potentially empty case.
\end{enumerate}
\end{enumerate}
\end{lem}
\begin{proof}
Let $h_{J}$ and $h_{K}$ be integers such that $J=\frac{1}{n^{k}}\left(C_{0}+h_{J}\right)$
and $K=\frac{1}{n^{k}}\left(C_{0}+h_{K}\right)$ and let 
\begin{align*}
J(p) & :=\frac{1}{n^{k+1}}\left(C_{0}+h_{J}n+d_{p}\right)\text{ and }\\
K(q) & :=\frac{1}{n^{k+1}}\left(C_{0}+h_{K}n+d_{q}\right)
\end{align*}
for $p,q=1,2,\ldots,m.$ Then the refinements of $J$ and $K$ are
$\bigcup{}_{p=1}^{m}J(p)$ and $\bigcup{}_{q=1}^{m}K(q).$ 

Suppose $J=K,$ then $h_{J}=h_{K}$ Hence $J(p)=K(q)+\frac{t_{k+1}}{n^{k+1}}$
iff $d_{p}=d_{q}+t_{k+1}$ and $J(p)=K(q)+\frac{t_{k+1}}{n^{k+1}}+\frac{1}{n^{k+1}}$
iff $d_{p}=d_{q}+t_{k+1}+1$. This establishes the interval case. 

Suppose $J=K+\frac{1}{n^{k}},$ then $h_{J}=h_{K}+1.$ So $J(p)=K(q)+\frac{t_{k+1}}{n^{k+1}}$
iff $n+d_{p}=d_{q}+t_{k+1}$ and $J(p)=K(q)+\frac{t_{k+1}}{n^{k+1}}+\frac{1}{n^{k+1}}$
iff $n+d_{p}=d_{q}+t_{k+1}+1$. This establishes the interval case.
\end{proof}
To describe our analysis of the sets $C_{k}\cap\left(C_{k}+\left\lfloor t\right\rfloor _{k}\right)$
we introduce appropriate terminology. 
\begin{itemize}
\item $C_{k}\cap\left(C_{k}+\left\lfloor t\right\rfloor _{k}\right)$ is
in the \emph{interval case,} if there exists an $n$-ary interval
$J\subset C_{k}$ in the interval case and no $n$-ary interval $K\subset C_{k}$
is in the potential interval case or simultaneous case.
\item $C_{k}\cap\left(C_{k}+\left\lfloor t\right\rfloor _{k}\right)$ is
in the \emph{potential interval case,} if there exists $J\subset C_{k}$
in the potential interval case and no $n$-ary interval $K\subset C_{k}$
is in the interval case or simultaneous case.
\item $C_{k}\cap\left(C_{k}+\left\lfloor t\right\rfloor _{k}\right)$ is
in the \emph{simultaneous case,} if there exist $J_{I},J_{P}\subset C_{k}$
such that $J_{I}$ is in the interval case and $J_{P}$ is in the
potential interval case.
\item $C_{k}\cap\left(C_{k}+\left\lfloor t\right\rfloor _{k}\right)$ is
in the \emph{irrecoverable case,} if $J$ is in the empty or potentially
empty case for all $n$-ary intervals $J\subset C_{k}$.
\end{itemize}
Our next goal is to introduce a function whose values determine whether
$C_{k}\cap\left(C_{k}+\left\lfloor t\right\rfloor _{k}\right)$ is
in the interval, potential interval, simultaneous, or irrecoverable
case. Since $C_{0}\cap\left(C_{0}+\left\lfloor t\right\rfloor _{0}\right)=\left[0,1\right]$,
then we begin in the interval case and can examine transitions using
inductively. The following constructions are motivated by Lemma \ref{Sec-4-Lemma:Transition}.
Let $i:=\sqrt{-1}$ and let
\[
\xi:\left\{ 0,\pm1,i\right\} \times\left\{ 0,1,\ldots,n-1\right\} \rightarrow\left\{ 0,\pm1,\pm i\right\} 
\]
be determined by
\begin{align*}
\xi\left(0,h\right) & :=0\\
\xi\left(1,h\right) & :=\begin{cases}
1 & \text{if }h\text{ is in }\Delta\text{ but not in }\Delta-1\\
-1 & \text{if }h\text{ is in }\Delta-1\text{ but not in }\Delta\\
i & \text{if }h\text{ is both in }\Delta\text{ and }\Delta-1\\
0 & \text{otherwise}
\end{cases}\\
\xi\left(-1,h\right) & :=\begin{cases}
-1 & \text{if }h\text{ is in }n-\Delta\text{ but not in }n-\Delta-1\\
1 & \text{it }h\text{ is in }n-\Delta-1\text{ but not in }n-\Delta\\
-i & \text{if }h\text{ is both in }n-\Delta\text{ and in }n-\Delta-1\\
0 & \text{otherwise}
\end{cases}\\
\xi\left(i,h\right) & :=\begin{cases}
-i & \text{if }h\text{ is in }\Delta\cup\left(n-\Delta\right)\text{ but not in }\left(\Delta-1\right)\cup\left(n-\Delta-1\right)\\
i & \text{if }h\text{ is in }\left(\Delta-1\right)\cup\left(n-\Delta-1\right)\text{ but not in }\Delta\cup\left(n-\Delta\right)\\
1 & \text{if }h\text{ is both in }\Delta\cup\left(n-\Delta\right)\text{ and in }\left(\Delta-1\right)\cup\left(n-\Delta-1\right)\\
0 & \text{otherwise}.
\end{cases}
\end{align*}
The function $\xi\left(z,h\right)$ is completely determined by $D$
and $n$. Let $\sigma_{t}:\mathbb{N}_{0}\to\left\{ 0,\pm1,i\right\} $
be determined by
\begin{align*}
\sigma_{t}(0) & :=1\text{ and inductively }\\
\sigma_{t}\left(k+1\right) & :=\xi\left(\sigma_{t}\left(k\right),t_{k+1}\right)\cdot\sigma_{t}\left(k\right)\text{ for }k\geq0.
\end{align*}
By construction of $\xi$ we have $\sigma_{t}\left(k\right)\in\left\{ 0,\pm1,i\right\} $
for all $k\geq0.$ 
\begin{lem}
\label{Sec-4-Lem:Sigma-property}Let $t=0._{n}t_{1}t_{2}\cdots$ be
some point in $[0,1].$ Then $C_{k}\cap\left(C_{k}+\left\lfloor t\right\rfloor _{k}\right)$
is in the interval case iff $\sigma_{t}\left(k\right)=1$, the potential
interval case iff $\sigma_{t}\left(k\right)=-1$, the simultaneous
case iff $\sigma_{t}\left(k\right)=i$, and the irrecoverable case
iff $\sigma_{t}\left(k\right)=0$.\end{lem}
\begin{proof}
This is a simple consequence of Lemma \ref{Sec-4-Lemma:Transition}
and our construction of $\sigma.$ 
\end{proof}
We now show that $D$ is sparse iff every $t\geq0$ in $F$ has an
$n$-ary representation such that for all $k\geq0$ the set $C_{k}\cap\left(C_{k}+\left\lfloor t\right\rfloor _{k}\right)$
is either in the interval case or in the potential interval case. 
\begin{thm}
\label{Sec-4-Thm:Sparcity_Requirement}Let $C=C_{n,D}$ be a deleted
digits Cantor set. Then 
\[
F^{+}=\left\{ t\in[0,1]\mid\sigma_{t}\left(k\right)=\pm1\text{ for all }k\in\mathbb{N}\right\} 
\]
iff $D$ is sparse. \end{thm}
\begin{proof}
Suppose $D$ is sparse, then $\Delta\cap\left(\Delta-1\right)=\varnothing$
and $\left(n-\Delta\right)\cap\left(n-\Delta-1\right)=\varnothing$.
Hence, our construction of $\xi$ and $\sigma$ shows that $\sigma_{t}\left(k\right)\in\left\{ 0,\pm1\right\} $
for all $k$ and all $t\in F^{+}.$ We must show that $\sigma_{t}\left(k\right)\neq0$
for all $k$ and all $t\in F^{+}.$

Suppose $t\in F^{+}$ such that $\sigma_{t}\left(k\right)=0$ for
some $k$. By Lemma \ref{Sec-4-Lem:Sigma-property} all $n$-ary intervals
in $C_{k}$ are in the potentially empty or the empty case. Since
$t\in F^{+}$ at least one $n$-ary interval, $J$ say, in $C_{k}$
is in the potentially empty case and $t_{j}=0$ for all $j>k.$ Since
$0\in\Delta$ it follows from the construction of $\sigma$ that $t\neq0.$
Hence, there is a $k\geq1$ such that $t_{k}>0$ and $t_{j}=0$ for
all $j>k.$ Let $s_{j}=t_{j}$ when $j<k,$ $s_{k}=t_{k}-1,$ and
$s_{j}=d_{m-1}$ for all $j>k.$ Then $t=0.s_{1}s_{2}\cdots.$ We
must show that $\sigma_{s}(j)\neq0$ for all $j$. Now $\sigma_{s}(j)=\sigma_{t}(t)\in\left\{ \pm1\right\} $
for all $j<k.$ Hence it remains to consider $j\geq k.$ 

The potentially empty cases in $C_{k}\cap\left(C_{k}+\left\lfloor t\right\rfloor _{k}\right)$
are interval cases in $C_{k}\cap\left(C_{k}+\left\lfloor t\right\rfloor _{k}-\frac{1}{n^{k}}\right).$
Some of the empty cases in $C_{k}\cap\left(C_{k}+\left\lfloor t\right\rfloor _{k}\right)$
may give potentially empty cases in $C_{k}\cap\left(C_{k}+\left\lfloor t\right\rfloor _{k}-\frac{1}{n^{k}}\right),$
but they cannot give interval cases in $C_{k}\cap\left(C_{k}+\left\lfloor t\right\rfloor _{k}-\frac{1}{n^{k}}\right).$
Consequently, $\sigma_{s}(k)=1.$ 

Since $C_{j}\cap\left(C_{j}+\left\lfloor t\right\rfloor _{j}\right)=C_{j}\cap\left(C_{j}+t\right)$
for all $j\geq k$ and $t\in F$ it follows from (\ref{Sec-2-eq:C_t-as-intersection-of- C_ks})
that $C_{j}\cap\left(C_{j}+\left\lfloor t\right\rfloor _{j}\right)$
is non-empty for all $j\geq k.$ 

Since $t=0._{n}t_{1}\cdots t_{k}$ is in $F^{+}$ and no intervals
in $C_{k}$ are in the interval case Theorem \ref{Sec-3-thm:Finite-Representations}
implies $d_{m}=n-1.$ Since $\sigma_{s}(k)=1$ and $s_{j}=d_{m}=n-1\in\Delta,$
it follows from Lemma \ref{Sec-4-Lemma:Transition} that $\sigma_{s}(j)=1$
for all $j>k.$

Conversely, suppose $D$ is not sparse, then $\Delta\cap\left(\Delta-1\right)\neq\varnothing$.
Let $\delta\in\Delta\cap\left(\Delta-1\right)$. Consider $t:=\frac{\delta}{n}$.
Then $\sigma_{t}\left(1\right)=i.$ Hence $C_{1}\cap\left(C_{1}+\left\lfloor t\right\rfloor _{1}\right)=C_{1}\cap\left(C_{1}+t\right)$
contains at least one $n$-ary interval $J$ which is in the interval
case.  The $n$-ary intervals in intervals in $C_{1}\cap\left(C_{1}+t\right)$
refine to $\frac{1}{n}\left(C+h\right)$ for some integer $h$. By
(\ref{Sec-2-eq:C_t-as-intersection-of- C_ks}) $\frac{1}{n}\left(C+h\right)\subseteq C\cap\left(C+t\right).$
In particular, $C\cap\left(C+t\right)\neq\varnothing$ so that $t\in F^{+}$.
\end{proof}
Theorem \ref{Sec-4-Thm:Sparcity_Requirement} shows that the simultaneous
case does not occur when $D$ is sparse. In particular, we have established
Proposition \ref{Sec-1-prop:Either-or}. 

In the following two lemmas we establish two key results required
to establish Theorem \ref{Sec-1-thm:estimate-measure}. In Lemma \ref{Sec-4-Lem:Interval_Count}
we show that $\mu_{t}\left(k\right)$ counts the number of $n$-ary
intervals of $C_{k}$ in either the interval or the potential interval
case. In Lemma \ref{Sec-4-lem:Nothing_Goes_Away_1} we show that the
intervals counted by $\mu_{t}(k)$ have points in common with $C\cap\left(C+t\right),$
hence that we do not ``over'' count. 
\begin{lem}
\label{Sec-4-Lem:Interval_Count}Let $C=C_{n,D}$ be given. Suppose
$t\in F^{+}$ does not admit a finite $n$-ary representation and
$\sigma_{t}(k)=\pm1$ for all $k\geq0$. Then $C_{k}\cap\left(C_{k}+t\right)$
is a union of $\mu_{t}(k)$ intervals, each of length 
\begin{align*}
\ell_{k}:= & \begin{cases}
\frac{1}{n^{k}}-\left(t-\left\lfloor t\right\rfloor _{k}\right) & \text{ when }\sigma_{t}\left(k\right)=1\\
t-\left\lfloor t\right\rfloor _{k} & \text{ when }\sigma_{t}\left(k\right)=-1
\end{cases}.
\end{align*}
\end{lem}
\begin{proof}
Let $t\in F^{+}$ be given. Suppose $t$ does not admit a finite $n$-ary
representation and $\sigma_{t}(k)=\pm1$ for all $k$. Every $n$-ary
interval in $C_{k}$ is either in the interval, the potential, interval,
or the potentially empty case. By Lemma \ref{Sec-3-Lem:empty-can-be-ignored},
if $J$ is an $n$-ary interval in $C_{k}$ that is in the potentially
empty or the empty case, then $J\cap\left(C_{k}+t\right)$ is empty.
Hence, it is sufficient to consider $n$-ary intervals in $C_{k}$
that either are in the interval or the potential intervals case. By
definition of $\sigma_{t}$, no $n$-ary interval in $C_{k}$ is both
in the interval and the potential interval case. 

Since the length of the intervals is determined by Lemma \ref{Sec-4-Lem:Sigma-property}
and Remark \ref{Sec-3-Remark:lengts-of-intersections}, we only need
to show that $C_{k}\cap\left(C_{k}+t\right)$ contains $\mu_{t}(k)$
intervals for $k\geq0.$ Since $C_{0}\cap\left(C_{0}+t\right)=[t,1]$
is one interval and $\mu_{t}(0)=1$, the claim holds for $k=0.$ 

Assume the claim holds for some integer $k\geq0.$ Then $C_{k}\cap\left(C_{k}+t\right)$
consists of $\mu_{t}(k)$ intervals. Suppose $\sigma_{t}(k)=1.$ Then
$C_{k}$ contains $\mu_{t}(k)$ $n$-ary intervals $J_{j}$ in the
interval case and no intervals in the potential interval case. Since
$t\in F^{+}$ it follows from part (1) of Lemma \ref{Sec-4-Lemma:Transition}
and Lemma \ref{Sec-3-Lem:empty-can-be-ignored} that $t_{k+1}\in\Delta$
or $t_{k+1}\in\Delta-1.$ If $t_{k+1}\in\Delta,$ then each $J_{j}$
gives $\#D\cap\left(D+t_{k+1}\right)$ intervals in $C_{k}\cap\left(C_{k}+t\right)$
by part (1)(a) of Lemma \ref{Sec-4-Lemma:Transition} and Remark \ref{Sec-3-Remark:lengts-of-intersections}.
Hence $C_{k+1}\cap\left(C_{k+1}+t\right)$ contains $\mu_{t}(k)\cdot\#D\cap\left(D+t_{k+1}\right)$
intervals. On the other hand, if $D\cap\left(D+t_{k+1}+1\right)$
is nonempty, then $t_{k+1}$ is an element of $\Delta\cap\left(\Delta-1\right)$
which contradicts the assumption that $\sigma_{t}\left(k+1\right)\neq i$.
Hence $\left(D-t_{k+1}\right)\cap\left(D\cup\left(D+1\right)\right)=D\cap\left(D+t_{k+1}\right).$
Consequently, $\mu_{t}(k+1)=\mu_{t}(k)\cdot\#D\cap\left(D+t_{k+1}\right)$
by the definition of $\mu_{t}.$ The case $t_{k+1}\in\Delta-1$ is
similar to $t_{k+1}\in\Delta.$

The case $\sigma_{t}(k)=-1$ is handled using arguments similar to
those used for $\sigma_{t}(k)=1$ above, replacing $\Delta$ by $n-\Delta$
and $\Delta-1$ by $n-\Delta-1.$ \end{proof}
\begin{lem}
\label{Sec-4-lem:Nothing_Goes_Away_1}Let $C=C_{n,D}$ be given. Suppose
$t\in F^{+}$ does not admit a finite $n$-ary representation and
$\sigma_{t}(k)=\pm1$ for all $k\geq0$. For each $k$, every $n$-ary
interval of $C_{k}$ in the interval or potential interval case contains
points of $C\cap\left(C+t\right)$.\end{lem}
\begin{proof}
Let $J_{0}=\frac{1}{n^{k}}\left(C_{0}+h\right)$ be an $n$-ary interval
of $C_{k}.$ Suppose $J_{0}$ is in the interval case. Let $x_{k}$
be the right hand endpoint of $J_{0}.$ Since $0<t-\left\lfloor t\right\rfloor _{k}<\frac{1}{n^{k}}$
and $J_{0}$ has length $\frac{1}{n^{k}}$ then $x_{k}\in J_{0}\cap\left(J_{0}+t-\left\lfloor t\right\rfloor _{k}\right).$
Now $J_{0}\cap\left(J_{0}+t-\left\lfloor t\right\rfloor _{k}\right)\subseteq C_{k}\cap\left(C_{k}+t\right)$
follows from $J_{0}\subseteq C_{k}+\left\lfloor t\right\rfloor _{k}.$
Consequently, $x_{k}$ is in $C_{k}\cap\left(C_{k}+t\right).$ 

Supposing $J_{0}$ is in the potential interval case and $x_{k}$
be the left hand endpoint of $J_{0},$ an argument similar to the
one above shows that $x_{k}$ is in $C_{k}\cap\left(C_{k}+t\right).$ 

Suppose $J_{0}$ is in the interval case. Then $\sigma_{t}\left(k\right)=1$
by assumption and all $n$-ary intervals in $C_{k}$ are either in
the interval case or one of the empty cases. Since $t\in F^{+}$ it
follows from Lemma \ref{Sec-4-Lemma:Transition} and Lemma \ref{Sec-3-Lem:empty-can-be-ignored}
that at least one subinterval $J_{1}$ in the refinement of $J_{0}$
is either in the interval or the potential interval case. Similarly,
if $J_{0}$ is in the potential interval case it follows that one
of the subintervals $J_{1}$ in the refinement of $J_{0}$ is in the
interval or potential interval case. 

By induction we get a sequence $x_{j}$ of points and a sequence of
intervals $J_{j}$ such that $J_{j+1}\subset J_{j}$ and $x_{j}\in C_{j}\cap\left(C_{j}+t\right)\subseteq J_{j}.$
By the nested interval theorem $x_{j}\to x\in\bigcap J_{j}\subset J_{0}.$
By (\ref{Sec-2-eq:C_t-as-intersection-of- C_ks}) $x\in C\cap\left(C+t\right).$ 
\end{proof}
Theorem \ref{Sec-4-Thm:Sparcity_Requirement} shows that the assumptions
of the previous Lemmas are met whenever $t$ does not admit finite
$n$-ary representation and $D$ is sparse. Example \ref{ex: NGA-Required}
demonstrates we may ``over'' count when $t$ does not meet the $\sigma_{t}\left(k\right)=\pm1$
requirement.

\section{\label{sec:Calculating-Measure}Estimating the Hausdorff Measure
of $C\cap\left(C+t\right)$}

Let $\mathscr{H}^{s}\left(K\right)$ denote the $s$-dimensional Hausdorff
measure of a compact set $K$ and let $\left|K\right|:=\sup\left\{ \left|x-y\right|\mid x,y\in K\right\} $
denote the diameter. Given $\varepsilon>0$, a collection of closed
intervals $\left\{ U_{\alpha}\right\} $ is an\emph{ $\varepsilon$-cover}
of $K$ if $K\subset\bigcup U_{\alpha}$ and $\varepsilon>\left|U_{\alpha}\right|>0$.
Define
\[
\mathscr{H}_{\varepsilon}^{s}\left(K\right):=\inf\left\{ \sum\left|U_{\alpha}\right|^{s}\right\} 
\]
 to be the approximation to the Hausdorff measure of $K$ by $\varepsilon$-covers
so that
\begin{equation}
\mathscr{H}^{s}\left(K\right)=\lim_{\varepsilon\to0}\mathscr{H}_{\varepsilon}^{s}\left(K\right).\label{eq:Hausdorff measure}
\end{equation}

The approximating measure $\mathscr{H}_{\varepsilon}^{s}\left(K\right)$
can be equivalently defined using a collection of arbitrary open or
closed sets, each having appropriate diameter. The closed intervals
definition is natural for this paper based on the construction of
$C\cap\left(C+t\right)$.

The Hausdorff dimension of $C$ is $\log_{n}\left(m\right)$ and $0<\mathscr{H}^{\log_{n}\left(m\right)}\left(C\right)<\infty$
since $C$ is self-similar by \cite{Hut81}. Since $C\cap\left(C+t\right)\subseteq C$,
then $0\le\dim\left(C\cap\left(C+t\right)\right)\le\log_{n}\left(m\right)$
for any real $t$ and if $0<\dim\left(C\cap\left(C+t\right)\right)<\log_{n}\left(m\right)$
then $t$ does not admit finite $n$-ary representation by Theorem
\ref{Sec-3-thm:Finite-Representations}. Our goal is to estimate the
Hausdorff measure of $C\cap\left(C+t\right)$.

\subsection{Infinite $n$-ary representations.}

We use the counting method of Lemma \ref{Sec-4-Lem:Interval_Count}
to estimate the Hausdorff measure of $C\cap\left(C+t\right)$ whenever
$t$ does not admit finite $n$-ary representation.
\begin{thm}
\label{thm:m^-bL<H<L}Let $C=C_{n,D}$ be given. Suppose $t$ is an
element of $F^{+}$ which does not admit finite $n$-ary representation
and $\sigma_{t}\left(k\right)=\pm1$ for all $k$. If $L_{t}:=\liminf_{k\to\infty}\left\{ m^{\nu_{t}\left(k\right)-k\cdot\beta_{t}}\right\} $
and $s:=\beta_{t}\log_{n}\left(m\right)$, then 
\[
m^{-\beta_{t}}\cdot L_{t}\le\mathscr{H}^{s}\left(C\cap\left(C+t\right)\right)\le L_{t}.
\]
\end{thm}
\begin{proof}
We begin by showing $\mathscr{H}^{s}\left(C\cap\left(C+t\right)\right)\le L_{t}$.
Let $N\in\mathbb{N}_{0}$ be given and $k\ge N$ be arbitrary so that
$n^{-N}\ge n^{-k}$. 

Lemma \ref{Sec-4-Lem:Interval_Count} shows that $C_{k}\cap\left(C_{k}+\left\lfloor t\right\rfloor _{k}\right)$
consists of $m{}^{\nu_{t}\left(k\right)}$ closed $n$-ary intervals
which cover $C\cap\left(C+t\right)$. Let $V_{i}$ denote the $i^{\text{th}}$
such interval of length $\frac{1}{n^{k}}$ so that $\left\{ V_{i}\right\} _{i=1}^{m^{\nu_{t}\left(k\right)}}$
is the collection of intervals or potential intervals chosen from
$C_{k}\cap\left(C_{k}+\left\lfloor t\right\rfloor _{k}\right)$. Then
\begin{align}
\mathscr{H}_{n^{-N}}^{s}\left(C\cap\left(C+t\right)\right) & \le\sum_{i=1}^{m^{\nu_{t}\left(k\right)}}\left|V_{i}\right|^{s}=m^{\nu_{t}\left(k\right)}\cdot\left(\frac{1}{n^{k}}\right)^{\beta_{t}\log_{n}\left(m\right)}=m^{\nu_{t}\left(k\right)-k\cdot\beta_{t}}.\label{eq:L_t-upper-bound}
\end{align}

Since $k\ge N$ is arbitrary, then $\mathscr{H}_{n^{-N}}^{s}\left(C\cap\left(C+t\right)\right)\le\liminf_{k\to\infty}\left\{ m^{\nu_{t}\left(k\right)-k\cdot\beta_{t}}\right\} $
and $\mathscr{H}^{s}\left(C\cap\left(C+t\right)\right)\le L_{t}$
by equation (\ref{eq:Hausdorff measure}). Thus, if $L_{t}=0$ then
$\mathscr{H}^{s}\left(C\cap\left(C+t\right)\right)=0$ and we are
finished. 

Suppose $0<L_{t}<\infty$. Then for arbitrarily small $\delta>0$,
there exists $N\left(\delta\right)\in\mathbb{N}$ such that $L_{t}-\delta\le m^{\nu_{t}\left(k\right)-k\cdot\beta_{t}}$
for all $k\ge N\left(\delta\right)$. Let $\varepsilon=n^{-N\left(\delta\right)}$. 

Let $\left\{ U_{\alpha}\right\} $ be an arbitrary closed $\varepsilon$-cover
of $C\cap\left(C+t\right)$. By compactness of $C\cap\left(C+t\right)$,
there exists a finite subcover $\left\{ U_{i}\right\} _{i=1}^{r}$
for some integer $r$. For each $1\le i\le r$, let $h_{i}$ denote
the integer satisfying 
\[
\left(\frac{1}{n}\right)^{h_{i}+1}\le\left|U_{i}\right|<\left(\frac{1}{n}\right)^{h_{i}}.
\]

Let $k\ge\max\left\{ h_{i}+1\mid1\le i\le r\right\} $ be arbitrary.
For each $1\le i\le r$, define $\mathcal{U}_{i}$ to be the collection
of $n$-ary intervals $J\subset C_{k}\cap\left(C_{k}+\left\lfloor t\right\rfloor _{k}\right)$
such that $J$ is in either the potential interval or the interval
case and $J\cap U_{i}\neq\varnothing$. Since $\sigma_{t}\left(k\right)=\pm1$
by assumption, then each $J\in\mathcal{U}_{i}$ contains points of
$C\cap\left(C+t\right)$ by Lemma \ref{Sec-4-lem:Nothing_Goes_Away_1}.
Thus, $\bigcup_{i=1}^{r}\mathcal{U}_{i}=C_{k}\cap\left(C_{k}+\left\lfloor t\right\rfloor _{k}\right)$. 

For any $j$, the set $C_{j}\cap\left(C_{j}+\left\lfloor t\right\rfloor _{j}\right)$
contains $m^{\nu_{t}\left(j\right)}$ intervals which all transition
the same way, then each interval $K\subset C_{h_{i}}\cap\left(C_{h_{i}}+\left\lfloor t\right\rfloor _{h_{i}}\right)$
transitions to $m^{\nu_{t}\left(k\right)-\nu_{t}\left(h_{i}\right)}$
intervals or potential intervals of $C_{k}\cap\left(C_{k}+\left\lfloor t\right\rfloor _{k}\right)$. 

If there exists an $n$-ary interval $J$ such that both $J$ and
$J-\frac{1}{n^{h_{i}}}$ are intervals in $C_{h_{i}}\cap\left(C_{h_{i}}+\left\lfloor t\right\rfloor _{h_{i}}\right)$
then $J$ is in both the interval and potential interval case. However,
$\sigma_{t}\left(h_{i}\right)=\sqrt{-1}$ by Lemma \ref{Sec-4-Lem:Sigma-property},
which contradicts our assumption. Thus, any pair of $n$-ary intervals
of $C_{h_{i}}\cap\left(C_{h_{i}}+\left\lfloor t\right\rfloor _{h_{i}}\right)$
are separated by at least $\frac{1}{n^{h_{i}}}$. Due to the diameter
$\frac{1}{n^{h_{i}}}>\left|U_{i}\right|$, each $U_{i}$ intersects
at most one interval of $C_{h_{i}}\cap\left(C_{h_{i}}+\left\lfloor t\right\rfloor _{h_{i}}\right)$.
Thus, 
\[
m^{\nu_{t}\left(k\right)}=\#\left(\bigcup_{i=1}^{r}\mathcal{U}_{i}\right)\le\sum_{i=1}^{r}\#\mathcal{U}_{i}\le\sum_{i=1}^{r}m^{\nu_{t}\left(k\right)-\nu_{t}\left(h_{i}\right)}.
\]

Hence, $1\le\sum_{i=1}^{r}m^{-\nu_{t}\left(h_{i}\right)}$. Furthermore,
$\left(L_{t}-\delta\right)m^{-\nu_{t}\left(h_{i}\right)}\le m^{-h_{i}\cdot\beta_{t}}$
since $h_{i}\ge N\left(\delta\right)$ by choice of $\varepsilon$.
\begin{align}
\sum_{i=1}^{r}\left|U_{i}\right|^{s} & \ge\sum_{i=1}^{r}\left(\frac{1}{n}\right)^{\left(h_{i}+1\right)\beta_{t}\log_{n}\left(m\right)}\nonumber \\
 & \ge m^{-\beta_{t}}\cdot\sum_{i=1}^{r}m^{-\beta_{t}\cdot h_{i}}\nonumber \\
 & \ge m^{-\beta_{t}}\left(L_{t}-\delta\right)\sum_{i=1}^{r}m^{-\nu_{t}\left(h_{i}\right)}\nonumber \\
 & \ge m^{-\beta_{t}}\left(L_{t}-\delta\right).\label{eq:m^(-b)*L}
\end{align}

Since $\left\{ U_{\alpha}\right\} $ is an arbitrary $\varepsilon$-cover
of $C\cap\left(C+t\right)$ then $\mathscr{H}_{\varepsilon}^{s}\left(C\cap\left(C+t\right)\right)\ge m^{-\beta_{t}}\left(L_{t}-\delta\right)$.
Furthermore, $\varepsilon=n^{-N\left(\delta\right)}\to0$ as $\delta\to0$
so that 
\[
m^{-\beta_{t}}L_{t}=\lim_{\delta\to0}\left(m^{-\beta_{t}}\left(L_{t}-\delta\right)\right)\le\lim_{\varepsilon\to0}\mathscr{H}_{\varepsilon}^{s}\left(C\cap\left(C+t\right)\right)=\mathscr{H}^{s}\left(C\cap\left(C+t\right)\right).
\]

Suppose $L_{t}=\infty$. Then for each $j\in\mathbb{N}$ there exists
$N\left(j\right)\in\mathbb{N}$ such that $j\le m^{\nu_{t}\left(k\right)-k\cdot\beta_{t}}$
for all $k\ge N\left(j\right)$. Choose $\varepsilon$ such that $n^{-N\left(\left\lceil m^{\beta_{t}}\cdot j\right\rceil \right)}>\varepsilon>0$.
Thus we can replace $\left(L_{t}-\delta\right)$ by $\left\lceil m^{\beta_{t}}\cdot j\right\rceil $
in equation (\ref{eq:m^(-b)*L}) so that
\begin{align*}
\sum_{i=1}^{r}\left|U_{i}\right|^{s} & \ge m^{-\beta_{t}}\cdot\left\lceil m^{\beta_{t}}\cdot j\right\rceil \sum_{i=1}^{r}m^{-\nu_{t}\left(h_{i}\right)}\ge j.
\end{align*}

Hence, $\mathscr{H}^{s}\left(C\cap\left(C+t\right)\right)\ge\lim_{j\to\infty}\left(j\right)=\infty$.
\end{proof}
Theorem \ref{thm:m^-bL<H<L} shows that $C\cap\left(C+t\right)$ is
an $s$-set \cite{Fal85} whenever $0<L_{t}<\infty$ and $C\cap\left(C+t\right)$
is not self-similar for any $t$ such that $L_{t}$ is either zero
or infinite. Furthermore, if $C=C_{n,D}$ is sparse and $t\in F^{+}$
does not admit finite $n$-ary representation, then $m^{-\beta_{t}}\cdot L_{t}\le\mathscr{H}^{s}\left(C\cap\left(C+t\right)\right)\le L_{t}$
by Theorem \ref{Sec-4-Thm:Sparcity_Requirement}.
\begin{rem}
\label{Sec-5:Remark-smaller-upper-bound}The proof of Theorem \ref{thm:m^-bL<H<L}
calculates the upper bound $L_{t}$ using the collection of $n$-ary
intervals chosen from $C_{k}\cap\left(C_{k}+\left\lfloor t\right\rfloor _{k}\right)$.
When $D$ is sparse, then $C_{k}\cap\left(C_{k}+t\right)$ consists
of $m^{\nu_{t}\left(k\right)}$ intervals of length $\ell_{k}\le\frac{1}{n^{k}}$
which also cover $C\cap\left(C+t\right)$ by Lemma \ref{Sec-4-Lem:Interval_Count}.
Choosing this cover, we can replace $\frac{1}{n^{k}}$ by $\ell_{k}$
in equation (\ref{eq:L_t-upper-bound}) and define $\widetilde{L}_{t}:=\liminf_{k\to\infty}\left\{ m^{\nu_{t}\left(k\right)}\left(\ell_{k}\right)^{\beta_{t}\log_{n}\left(m\right)}\right\} $
so that 
\[
\mathscr{H}^{s}\left(C\cap\left(C+t\right)\right)\le\widetilde{L}_{t}\le L_{t}.
\]

This may calculate a more accurate upper bound for the Hausdorff measure
of $C\cap\left(C+t\right)$, however it is more difficult to calculate
$\widetilde{L}_{t}$ since $\ell_{k}$ depends directly on $t$. Example
\ref{ex:< M_t =00003D L_t} shows that the Hausdorff measure may be
strictly smaller than $\widetilde{L}_{t}$.\end{rem}
\begin{cor}
\label{cor:H=00003Dbeta_t}Let $C=C_{n,D}$ be given. If $t\in F^{+}$
does not admit finite $n$-ary representation and $\sigma_{t}\left(k\right)=\pm1$
for all $k$, then the Hausdorff dimension of $C\cap\left(C+t\right)$
is $\beta_{t}\log_{n}\left(m\right)$.\end{cor}
\begin{proof}
The dimension is determined by Theorem \ref{thm:m^-bL<H<L} whenever
$0<L_{t}<\infty$. We need to show the result when $L_{t}$ is zero
or infinite. Let $\varepsilon>0$ be given and $\left\{ U_{i}\right\} _{i=1}^{r}$
an arbitrary $\varepsilon$-cover of $C\cap\left(C+t\right)$ as in
the proof of Theorem \ref{thm:m^-bL<H<L}. Let $N\left(\varepsilon\right)\in\mathbb{N}$
be such that $\varepsilon>n^{-N\left(\varepsilon\right)}$.

Suppose $L_{t}=\infty$. Choose an arbitrary value $\gamma$ such
that $\beta_{t}<\gamma$ and choose $\delta$ such that $\gamma-\beta_{t}>\delta>0$.
By definition of $\beta_{t}$ there exists a subsequence $\left\{ h_{j}\right\} $
and integer $M\left(\delta\right)$ such that $\frac{\nu_{t}\left(h_{j}\right)}{h_{j}}<\beta_{t}+\delta<\gamma$
for all $j\ge M\left(\delta\right)$. Then for any $j\ge\max\left\{ N\left(\varepsilon\right),M\left(\delta\right)\right\} $
we can replace $\beta_{t}$ by $\gamma$ in the proof of Theorem \ref{thm:m^-bL<H<L}
so that
\begin{align*}
\mathscr{H}_{\varepsilon}^{\gamma\log_{n}\left(m\right)}\left(C\cap\left(C+t\right)\right) & \le\liminf_{j\to\infty}\left\{ m^{\nu_{t}\left(h_{j}\right)-h_{j}\cdot\gamma}\right\} =\liminf_{k\to\infty}\left\{ m^{\left(\frac{\nu_{t}\left(h_{j}\right)}{h_{j}}-\gamma\right)h_{j}}\right\} \\
 & \le\liminf_{k\to\infty}\left\{ m^{\left(\beta_{t}+\delta-\gamma\right)h_{j}}\right\} =0.
\end{align*}

Since $\varepsilon>0$ is arbitrary, then $\mathscr{H}^{\gamma\log_{n}\left(m\right)}\left(C\cap\left(C+t\right)\right)=0$
for any $\gamma>\beta_{t}$. 

Suppose $L_{t}=0$. Choose an arbitrary value $\gamma$ such that
$0\le\gamma<\beta_{t}$. Let $\Gamma_{t}:=\liminf_{k\to\infty}\left\{ m^{\nu_{t}\left(k\right)-k\cdot\gamma}\right\} $.
Choose $\delta$ such that $\beta_{t}-\gamma>\delta>0$ and choose
$M\left(\delta\right)$ such that $\Gamma_{t}-\delta\le m^{\nu_{t}\left(k\right)-k\cdot\gamma}$
for all $k\ge M\left(\delta\right)$. Thus, we can replace $\beta_{t}$
by $\gamma$ and $L_{t}$ by $\Gamma_{t}$ in the proof of Theorem
\ref{thm:m^-bL<H<L} so that $\mathscr{H}^{\gamma\log_{n}\left(m\right)}\left(C\cap\left(C+t\right)\right)$
is infinite whenever $\Gamma_{t}=\infty$.

Since $m^{\left(\beta_{t}-\delta-\gamma\right)}>1$, then any $k\ge\max\left\{ N\left(\varepsilon\right),M\left(\delta\right)\right\} $,

\begin{align*}
m^{\nu_{t}\left(k\right)-k\cdot\gamma} & =m^{\left(\frac{\nu_{t}\left(k\right)}{k}-\gamma\right)k}\\
 & \ge m^{\left(\beta_{t}-\delta-\gamma\right)k}\\
 & \ge m^{\left(\beta_{t}-\delta-\gamma\right)N\left(\varepsilon\right)}.
\end{align*}
Hence, $\Gamma_{t}\ge\liminf_{N\left(\varepsilon\right)\to\infty}\left\{ m^{\left(\beta_{t}-\delta-\gamma\right)N\left(\varepsilon\right)}\right\} =\infty$
so that $\mathscr{H}^{\gamma\log_{n}\left(m\right)}\left(C\cap\left(C+t\right)\right)=\infty$
for any $0\le\gamma<\beta_{t}$.\end{proof}
\begin{cor}
\label{cor:F_b,A dense in F}Let $C=C_{n,D}$ be sparse and $\beta,y\in\mathbb{R}$
such that $0<\beta<1$ and $0<y<\infty$. Define
\[
F_{\beta,y}:=\left\{ x\mid m^{-2\beta}\cdot y\le\mathscr{H}^{\beta\cdot\log_{n}\left(m\right)}\left(C\cap\left(C+x\right)\right)\le y\right\} .
\]

Then $F_{\beta,y}$ is dense in $F$.\end{cor}
\begin{proof}
Choose $0<\beta<1$ and $0<y<\infty$. It is sufficient to show that
$F_{\beta,y}^{+}$ is dense in $F^{+}$. Let $t\in F^{+}$ and $\varepsilon>0$
be given. We will construct the necessary $x=0._{n}x_{1}x_{2}\ldots$.

Let $k\in\mathbb{N}$ such that $\varepsilon>\left(\frac{1}{n}\right)^{k-1}>0$.
Choose $x_{j}=t_{j}$ for all $1\le j\le k-1$ so that $\left|x-t\right|<\varepsilon$
regardless of any choice of remaining digits $x_{j}$ for $j\ge k$.
If $\sigma_{x}\left(k-1\right)=1$ then choose $x_{k}=0$ so that
$\sigma_{x}\left(k\right)=1$. Otherwise, if $\sigma_{x}\left(k-1\right)=-1$
then choose $x_{k}=n-d_{m}$ so that $\sigma_{x}\left(k\right)=1$.
Thus $\sigma_{x}\left(k\right)=1$ and we begin in the interval case.

Since $k$ is finite, then $0<m^{\nu_{t}\left(k\right)-k\beta}<\infty$.
If $x_{j}=0$ then $\mu_{x}\left(x_{j}\right)=m$ so that $\nu_{x}\left(j+1\right)=\nu_{x}\left(j\right)+1$
and $m^{\nu_{x}\left(j\right)-j\cdot\beta}<m^{\nu_{x}\left(j+1\right)-\left(j+1\right)\beta}$.
Similarly, if $x_{j}=d_{m}$ then $\mu\left(x_{j}\right)=1$ so that
$\nu_{x}\left(j+1\right)=\nu_{x}\left(j\right)$ and $m^{\nu_{x}\left(j\right)-j\cdot\beta}>m^{\nu_{x}\left(j+1\right)-\left(j+1\right)\beta}$.
For all $j\ge k$, choose the remaining digits of $x$ such that 
\[
x_{j+1}=\begin{cases}
0 & \text{ if }m^{\nu_{x}\left(j\right)-j\cdot\beta}\le y\\
d_{m} & \text{ if }m^{\nu_{x}\left(j\right)-j\cdot\beta}>y.
\end{cases}
\]

Thus, if $x_{j+1}=d_{m}$ then $m^{\nu_{x}\left(j+1\right)-\left(j+1\right)\cdot\beta}=m^{-\beta}m^{\nu_{x}\left(j\right)-j\cdot\beta}>y\cdot m^{-\beta}$
so that 
\[
y\cdot m^{-\beta}\le\liminf_{j\to\infty}\left\{ m^{\nu_{x}\left(j\right)-j\cdot\beta}\right\} \le y.
\]

Therefore, $y\cdot m^{-2\beta}\le\mathscr{H}^{\beta\log_{n}\left(m\right)}\left(C\cap\left(C+x\right)\right)\le y$
by Theorem \ref{thm:m^-bL<H<L}.
\end{proof}
It would be ideal to construct $x$ such that $L_{x}=y$ in the proof
of Corollary \ref{cor:F_b,A dense in F}, however this is not always
possible. Example \ref{ex:L_t is nowhere dense.} shows a class of
sparse Cantor sets $C_{n,D}$ such that $L_{t}$ is either infinite
or some element of a countable, nowhere dense subset of $\mathbb{R}$
for all $t\in F^{+}$.
\begin{example}
\label{ex:L_t is nowhere dense.}Let $n\ge3$ and $D=\left\{ 0,d\right\} $
be given for some $2\le d<n$ so that $C=C_{n,D}$ is sparse. Choose
$\beta=\frac{a}{b}$ for some integers $0\le a\le b$ and $b\neq0$.
Then $\frac{\mu_{t}\left(j\right)}{\mu_{t}\left(j-1\right)}=1,2$
for any $t\in F^{+}$ and $j\in\mathbb{N}_{0}$. Define $p_{k}:=\#\left\{ j\le k\mid\mu\left(j\right)=2\mu_{t}\left(j-1\right)\right\} $
and $q_{k}:=\#\left\{ j\le k\mid\mu_{t}\left(j\right)=\mu_{t}\left(j-1\right)\right\} $
for each $k$ so that $p_{k},q_{k}\in\mathbb{N}_{0}$ and $k=p_{k}+q_{k}$.
Thus, 
\begin{align*}
\nu_{t}\left(k\right)-k\beta=p_{k}-\left(p_{k}+q_{k}\right)\beta & =\frac{1}{b}\left(p_{k}b-a\left(p_{k}+q_{k}\right)\right)\in\frac{1}{b}\mathbb{Z}.
\end{align*}

If $\liminf_{k\to\infty}\left\{ \nu_{t}\left(k\right)-k\beta\right\} =-\infty$
then $L_{t}=0$ and if $\liminf_{k\to\infty}\left\{ \nu_{t}\left(k\right)-k\beta\right\} =\infty$
then $L_{t}=\infty$. Otherwise, any subsequence $\nu_{t}\left(k_{j}\right)-k_{j}\beta\to r$
is a bounded sequence of $\frac{1}{b}\mathbb{Z}$. Hence, if $L_{t}$
is finite then $L_{t}\in\left\{ 2^{r}\mid b\cdot r\in\mathbb{Z}\right\} $
and there is no real $x$ such that $0<L_{x}<\sqrt[b]{2}$ for this
choice of $C_{n,D}$.
\end{example}

\subsection{Finite $n$-ary representations.}

According to Theorem \ref{Sec-3-thm:Finite-Representations}, if $t\in F^{+}$
admits finite $n$-ary representation then $C\cap\left(C+t\right)$
is either finite, or a finite collection of sets $\frac{1}{n^{k}}\left(C+h_{j}\right)$.
Therefore, the Hausdorff $\log_{n}\left(m\right)$-dimensional measure
is either zero or can be expressed in terms of $\mathscr{H}^{s}\left(C\right)$
for $s:=\log_{n}\left(m\right)$.

The exact Hausdorff measure of many Cantor set in $\left[0,1\right]$
can be calculated by methods of \cite{AySt99,Ma86,Ma87}; this includes
deleted digits Cantor sets $C=C_{n,D}$. The proof of Theorem 8.6
in \cite{Fal85} estimates the Hausdorff measure of an arbitrary self-similar
set and gives the bounds $\frac{1}{3n}\le\mathscr{H}^{s}\left(C_{n,D}\right)\le1$.
The basic idea of the proof of Theorem \ref{thm:m^-bL<H<L} leads
to bounds on $\mathscr{H}^{s}\left(C_{n,D}\right),$ we include these
bounds for completeness. This is much simpler than the proof of Theorem
\ref{thm:m^-bL<H<L} since the needed versions of Lemma \ref{Sec-4-Lem:Interval_Count}
and Lemma \ref{Sec-4-lem:Nothing_Goes_Away_1} are trivial. 
\begin{thm}
\label{thm:Finite Measure}Let $C=C_{n,D}$ be given and $s:=\log_{n}\left(m\right)$.
Then $\frac{1}{m}\le\mathscr{H}^{s}\left(C\right)\le1$.\end{thm}
\begin{proof}
Let $V_{i}$ denote the $i^{\text{th}}$ $n$-ary interval of $C_{k}$
so that $C_{k}=\bigcup_{i=1}^{m^{k}}V_{i}$ is a cover of $C$. Then
$\sum_{i=1}^{m^{k}}\left|V_{i}\right|^{s}=m^{k}\cdot\left(n^{-k}\right)^{\log_{n}\left(m\right)}=1$
for all $k\in\mathbb{N}_{0}$ so that $\mathscr{H}^{s}\left(C\right)\le\mathscr{H}_{n^{-k}}^{s}\left(C\right)\le1$.

The proof of the lower bound is similar to the proof of Theorem \ref{thm:m^-bL<H<L}
with minor variations. Let $\varepsilon>0$ be given and $\left\{ U_{i}\right\} _{i=1}^{r}$
be an arbitrary closed $\varepsilon$-cover of $C$ for some integer
$r$. For each $1\le i\le r$, let $h_{i}$ denote the integer satisfying
$n^{-h_{i}-1}\le\left|U_{i}\right|<n^{-h_{i}}$. 

Let $k\ge\max\left\{ h_{i}+1\mid1\le i\le r\right\} $ be arbitrary
and, for each $1\le i\le r$, define $\mathcal{U}_{i}$ to be the
collection of $n$-ary intervals $J$ selected from $C_{k}$ such
that $J\cap U_{i}\neq\varnothing$. Each $J\in\mathcal{U}_{i}$ contains
points of $C$ by the Nested Intervals Theorem so that $\bigcup_{i=1}^{r}\mathcal{U}_{i}=C_{k}$
and each interval $K\subset C_{h_{i}}$ contains $m^{k-h_{i}}$ $n$-ary
intervals of $C_{k}$.

Since $\frac{1}{n^{h_{i}}}>\left|U_{i}\right|$, then each $U_{i}$
intersects at most two intervals of $C_{h_{i}}$. Suppose $U_{i}$
intersects both $K$ and $K-\frac{1}{n^{h_{i}}}$ for some $n$-ary
interval $K\subset C_{h_{i}}$ and let $K\left(p\right)\subset C_{k}\cap K$
denote the $n$-ary subintervals of $K$ for $1\le p\le m^{k-h_{i}}$.
Note that if $U_{i}\cap K\left(p\right)\neq\varnothing$ for some
$p$ then $U_{i}\cap\left(K\left(p\right)-\frac{1}{n^{h_{i}}}\right)$
is empty unless $K\left(p\right)$ contains an endpoint of $U_{i}$
and $\left|U_{i}\right|>\frac{n-1}{n^{h_{i}+1}}$. Thus, $U_{i}$
intersects at most $m^{k-h_{i}}+1$ intervals of $C_{k}$ so that
$\#\mathcal{U}_{i}\le m^{k-h_{i}}+1$ for all $1\le i\le r$ and 
\[
m^{k}=\#\left(\bigcup_{i=1}^{r}\mathcal{U}_{i}\right)\le\sum_{i=1}^{r}\#\mathcal{U}_{i}\le\sum_{i=1}^{r}\left(m^{k-h_{i}}+1\right).
\]

Therefore, $1-r\cdot m^{-k}\le\sum_{i=1}^{r}m^{-h_{i}}$ so that
\begin{align}
\sum_{i=1}^{r}\left|U_{i}\right|^{s} & \ge\sum_{i=1}^{r}\left(\frac{1}{n}\right)^{\left(h_{i}+1\right)s}\ge\frac{1}{m}\cdot\sum_{i=1}^{r}m^{-h_{i}}\ge\frac{1}{m}\left(1-r\cdot m^{-k}\right).\label{eq:Lower Bound}
\end{align}

Since $\left\{ U_{i}\right\} _{i=1}^{r}$ is an arbitrary $\varepsilon$-cover
of $C$ and equation (\ref{eq:Lower Bound}) holds for any sufficiently
large $k$, then $\mathscr{H}_{\varepsilon}^{s}\left(C\right)\ge\lim_{k\to\infty}\left\{ \frac{1}{m}\left(1-rm^{-k}\right)\right\} =\frac{1}{m}$
for any $\varepsilon>0$. Hence, $\frac{1}{m}\le\mathscr{H}^{s}\left(C\right)$.
\end{proof}
Let $n=9$, $D=\left\{ 0,d,8\right\} $ for some integer $0<d<8$,
and $s:=\log_{9}\left(3\right)=\frac{1}{2}$. If $d=4$ then $D$
is uniform and $\mathscr{H}^{s}\left(C_{9,\left\{ 0,4,8\right\} }\right)=1$.
However, if $d=2$ then $D$ is regular and it is shown in example
\ref{ex:< M_t =00003D L_t} that $\mathscr{H}^{s}\left(C_{9,\left\{ 0,2,8\right\} }\right)<1$.
\begin{cor}
\label{Sec-5-cor:estimate-for-finite-t} Let $C=C_{n,D}$ be arbitrary,
$s:=\log_{n}\left(m\right)$, and $t\in F^{+}$ such that $t=0._{n}t_{1}t_{2}\cdots t_{k}$.
Then $C\cap\left(C+t\right)=A\cup B$ and the following hold:
\begin{enumerate}
\item If $A$ is nonempty, then $A=\bigcup_{j=1}^{a}\frac{1}{n^{k}}\left(C+h_{j}\right)$
for some integer $a$ and \textup{$\frac{a}{m^{k+1}}\le\mathscr{H}^{s}\left(C\cap C+t\right)\le\frac{a}{m^{k}}.$
In particular, if $D$ is sparse then $a=\mu_{t}\left(k\right)$.}
\item If $A$ is empty, then $\mathscr{H}^{0}\left(C\cap\left(C+t\right)\right)=\#B$.
If $D$ is sparse then $\#B=\mu_{t+n^{-k}}\left(k\right)+\mu_{t-n^{-k}}\left(k\right)$.
\end{enumerate}
\end{cor}
\begin{proof}
The general statements follow immediately from Theorem \ref{thm:Finite Measure}
and Theorem \ref{Sec-3-thm:Finite-Representations}. We only need
show the result when $D$ is sparse. Without loss of generality, assume
that $k$ is the minimal element of $\left\{ j\mid t=0._{n}t_{1}\cdots t_{j}\right\} $.

Suppose $A$ is nonempty and $s=\log_{n}\left(m\right)$. Since $\frac{1}{n^{j}}>t-\left\lfloor t\right\rfloor _{j}>0$
for any $1\le j<k$, we can apply Lemma \ref{Sec-3-Lem:empty-can-be-ignored}
so that $C_{k}\cap\left(C_{k}+\left\lfloor t\right\rfloor _{k}\right)=C_{k}\cap\left(C_{k}+t\right)$
consists of $\mu_{t}\left(k\right)$ disjoint intervals. Since each
such interval refines to $\frac{1}{n^{k}}\left(C+h_{j}\right)$ and
$\mathscr{H}^{s}\left(B\setminus A\right)=0$, it follows that $a=\mu_{t}\left(k\right)$.

Suppose $A$ is empty so that $B$ contains a finite number of isolated
points by definition of $F$. Any $n$-ary interval $J\subset C_{k}$
in the potential interval case is also an $n$-ary interval of $C_{k}+\left\lfloor t\right\rfloor _{k}+\frac{1}{n^{k}}$.
Thus, $J$ is in the interval case of $C_{k}\cap\left(C_{k}+\left\lfloor t\right\rfloor _{k}+\frac{1}{n^{k}}\right)$
and $B$ contains $\mu_{t+n^{-k}}\left(k\right)$ points corresponding
to potential intervals. Similarly, if $J\subset C_{k}\cap\left(C_{k}+\left\lfloor t\right\rfloor _{k}\right)$
is in the potentially empty case then $J$ is an interval case of
$C_{k}\cap\left(C_{k}+\left\lfloor t\right\rfloor _{k}-\frac{1}{n^{k}}\right)$
and $B$ contains $\mu_{t-n^{-k}}\left(k\right)$ points corresponding
to potentially empty cases. 

Since $d-d'\ge2$ for all $d,d'\in D\subset\Delta$, then no point
of $B$ can be in both the potential interval and potentially empty
cases. Hence, $\#B=\mu_{t+n^{-k}}\left(k\right)+\mu_{t-n^{-k}}\left(k\right)$.
\end{proof}

\section{Examples\label{sec:Examples}}

We use the results of the previous sections to estimate the Hausdorff
measure of $C\cap\left(C+t\right)$. The following examples demonstrate
when the Hausdorff measure is equal to both $\widetilde{L}_{t}$ and
$L_{t}$ (Example \ref{ex: =00003D M_t =00003D L_t}), equal to $\widetilde{L}_{t}$
but less than $L_{t}$ (Example \ref{ex: =00003D M_t < L_t}), or
less than both $\widetilde{L}_{t}$ and $L_{t}$ (Example \ref{ex:< M_t =00003D L_t}).
\begin{example}
\label{ex: =00003D M_t =00003D L_t}Let $C=C_{n,D}$ be sparse such
that $\mathscr{H}^{s}\left(C\right)=1$ for $s=\log_{n}\left(m\right)$.
This is true for the class of uniform sets such that $d_{m}=n-1$
by \cite{Fal85}. Choose $t=0._{n}t_{1}t_{2}\cdots t_{k}$ for some
$k$ such that $\sigma_{k}\left(t\right)=1$. Then $\nu_{t}\left(k+j\right)=\nu_{t}\left(k\right)+j$
for all $j\ge0$ and $\beta_{t}=1$ so that 
\[
L_{t}=\liminf_{j\to\infty}\left\{ m^{\nu_{t}\left(k+j\right)-\left(k+j\right)\beta_{t}}\right\} =m^{\nu_{t}\left(k\right)-k}.
\]
Since $C\cap\left(C+t\right)=\bigcup_{j}\frac{1}{n^{k}}\left(C+h_{j}\right)$
consists of $m^{\nu_{t}\left(k\right)}$ disjoint copies of $\frac{1}{n^{k}}C$,
then 
\[
\mathscr{H}^{s}\left(C\cap\left(C+t\right)\right)=m^{\nu_{t}\left(k\right)-k}\cdot\mathscr{H}^{s}\left(C\right)=L_{t}.
\]

\begin{example}
\label{ex: =00003D M_t < L_t}Let $C=C_{3,\left\{ 0,2\right\} }$
denote the Middle Thirds Cantor set and let $t:=0._{3}\overline{20}=\frac{3}{4}$.
Then $\nu_{t}\left(k\right)=\left\lfloor \frac{k+1}{2}\right\rfloor $
for all $k$ so that $\nu_{t}\left(2k\right)=k$ and $\nu_{t}\left(2k+1\right)=k+1$.
Thus, $\beta_{t}=\frac{1}{2}$ so that $\nu_{t}\left(2k\right)-2k\beta_{t}=0$
and $\nu_{t}\left(2k+1\right)-\left(2k+1\right)\beta_{t}=\frac{1}{2}$.
Hence, $L_{t}=\liminf_{k\to\infty}\left\{ 1,\sqrt{2},1,\ldots\right\} =1$.
\end{example}
Since $\ell_{2k}=\frac{1}{9^{k}}-\frac{1}{9^{k}}\left(\frac{3}{4}\right)=\frac{1}{4\cdot9^{k}}$
and $\ell_{2k+1}=\frac{1}{3\cdot9^{k}}-\frac{1}{3\cdot9^{k}}\left(\frac{1}{4}\right)=\frac{1}{4\cdot9^{k}}$,
then for $s:=\log_{9}\left(2\right)$,
\[
\widetilde{L}_{t}=\liminf_{k\to\infty}\left\{ 2^{\nu_{t}\left(k\right)-k}\left(\frac{1}{4}\right)^{s}\right\} =\liminf_{k\to\infty}\left\{ \left(\frac{9}{4}\right)^{s},\left(\frac{1}{4}\right)^{s},\left(\frac{9}{4}\right)^{s},\ldots\right\} =\left(\frac{1}{4}\right)^{s}.
\]

Therefore, $\mathscr{H}^{s}\left(C\cap\left(C+t\right)\right)\le\widetilde{L}_{t}<L_{t}$.
An upcoming paper, by the co-authors, shows that the Hausdorff measure
is exactly $4^{-s}$ for this choice of $C_{n,D}$ and $t$.
\begin{example}
\label{ex:< M_t =00003D L_t}Let $n=9$ and $D=\left\{ 0,2,8\right\} $
so that $C=C_{n,D}$ is regular. Choose $t:=0$ so that for all $k$,
$\nu_{t}\left(k\right)=k$, $\ell_{k}=\frac{1}{n^{k}}$, $\beta_{t}=1$,
and $\widetilde{L}_{t}=L_{t}=\liminf\left\{ m^{\nu_{t}\left(k\right)-k\beta_{t}}\right\} =1$.
Since $C\cap\left(C+t\right)=C$, we will show that $\mathscr{H}^{s}\left(C\right)<1$
for $s:=\log_{9}\left(3\right)=\frac{1}{2}$.
\end{example}
Let $\varepsilon>0$ be given and choose $k$ such that $\varepsilon>\frac{1}{n^{k-1}}$.
Let $J=\frac{1}{n^{k-1}}\left(C_{0}+h_{j}\right)$ be an arbitrary
$n$-ary interval of $C_{k-1}$. Then the refinement of $J$ consists
of three subintervals $J\left(1\right)=\frac{1}{n^{k}}\left(C_{0}+h_{j}n\right)$,
$J\left(2\right)=\frac{1}{n^{k}}\left(C_{0}+h_{j}n+2\right)$, and
$J\left(3\right)=\frac{1}{n^{k}}\left(C_{0}+h_{j}n+8\right)$. Choose
$U_{2j-1}=\frac{1}{n^{k}}\left(3C_{0}+h_{j}n\right)$ so that $J\left(1\right)\cup J\left(2\right)\subset U_{2j-1}$
and choose $U_{2j}=J\left(3\right)$. Since there are $3^{k-1}$ such
intervals $J$ and $\varepsilon>\left|J\right|>\left|U_{2j-1}\right|>\left|U_{2j}\right|$,
then the collection $\left\{ U_{j}\right\} _{j=1}^{2\cdot3^{k-1}}$
is an $\varepsilon$-cover of $C$. Therefore,
\begin{align*}
\mathscr{H}_{\varepsilon}^{s}\left(C\right) & \le\sum_{j=1}^{2\cdot3^{k-1}}\left|U_{j}\right|^{s}=\sum_{j=1}^{3^{k-1}}\left|U_{2j-1}\right|^{s}+\sum_{j=1}^{3^{k-1}}\left|U_{2j}\right|^{s}\\
 & =3^{k-1}\cdot\left(\frac{3}{9^{k}}\right)^{s}+3^{k-1}\cdot\left(\frac{1}{9^{k}}\right)^{s}=\frac{\sqrt{3}+1}{3}<\widetilde{L}_{t}.
\end{align*}

Since $\varepsilon>0$ is arbitrary, then $\frac{1}{3}\le\mathscr{H}^{s}\left(C\right)\le\frac{\sqrt{3}+1}{3}$
according to Theorem \ref{thm:Finite Measure}.
\end{example}
Theorem \ref{thm:m^-bL<H<L} shows that the Hausdorff measure of $C\cap\left(C+t\right)$
is equal to $L_{t}$ whenever $L_{t}$ is zero or infinite. In the
following example we construct $x,y\in F$ such that $L_{x}=\infty$
and $L_{y}=0$ so that the sets $C\cap\left(C+x\right)$ and $C\cap\left(C+y\right)$
are not self-similar.
\begin{example}
\label{ex:L_s=00003D0 and L_r=00003Dinfty}Let $n=11$, $D=\left\{ 0,7,10\right\} $,
and $t:=0._{11}\overline{70}$ so that $\nu_{t}\left(k\right)=\left\lfloor \frac{k+1}{2}\right\rfloor $
for all $k$. Thus, $\nu_{t}\left(2k\right)=k$ and $\nu_{t}\left(2k+1\right)=k+1$
so that $\beta_{t}=\frac{1}{2}$ and $s:=\frac{1}{2}\log_{11}\left(3\right)$.
Define $x:=0._{11}x_{1}x_{2}\ldots$ and $y:=0._{11}y_{1}y_{2}\ldots$
such that 
\begin{align*}
x_{k} & =\begin{cases}
0 & \text{ if }k=1+2j^{2}\text{ for some integer \ensuremath{j}}\\
t_{k} & otherwise
\end{cases}\\
y_{k} & =\begin{cases}
7 & \text{ if }k=2j^{2}\text{ for some integer \ensuremath{j}}\\
t_{k} & otherwise.
\end{cases}
\end{align*}

Since $\mu\left(t_{2j^{2}}\right)=1$ and $\mu\left(y_{2j^{2}}\right)=0$
for each integer $j$, and $\mu\left(t_{k}\right)=\mu\left(y_{k}\right)$
otherwise, then $\nu_{y}\left(2j^{2}\right)=\nu_{t}\left(2j^{2}\right)-j$
for each $j>0$. Thus, if $2j^{2}\le k<2\left(j+1\right)^{2}$ for
some $j$ then $\nu_{y}\left(k\right)=\nu_{t}\left(k\right)-j$ so
that $\beta_{y}=\beta_{t}=\frac{1}{2}$. Furthermore, 
\begin{align*}
L_{y} & \le\liminf_{j\to\infty}\left\{ 3^{\nu_{y}\left(2j^{2}\right)-\beta_{y}2j^{2}}\right\} =\liminf_{j\to\infty}\left\{ 3^{\nu_{t}\left(2j^{2}\right)-j-j^{2}}\right\} \\
 & =\liminf_{j\to\infty}\left\{ 3^{-j}\right\} =0.
\end{align*}

Therefore, $\mathscr{H}^{s}\left(C\cap\left(C+y\right)\right)=L_{y}=0$
by Theorem \ref{thm:m^-bL<H<L}.

Similarly, $\mu\left(t_{1+2j^{2}}\right)=0$ and $\mu\left(x_{1+2j^{2}}\right)=1$
for each integer $j$, and $\mu\left(t_{k}\right)=\mu\left(x_{k}\right)$
otherwise. Thus, $\nu_{x}\left(1+2j^{2}\right)=\nu_{t}\left(1+2j^{2}\right)+j$
for each $j>0$ and $\nu_{x}\left(k\right)=\nu_{t}\left(k\right)+j$
whenever $2j^{2}\le k<2\left(j+1\right)^{2}$. Therefore, $\beta_{x}=\beta_{t}=\frac{1}{2}$
and for each $k$, 
\begin{align*}
3^{\nu_{x}\left(k\right)-\beta_{x}\left(k\right)} & =3^{\nu_{t}\left(k\right)+j-\beta_{x}\left(k\right)}\ge3^{\frac{1}{2}k+j-\frac{1}{2}k}=3^{j}.
\end{align*}

Hence, $\mathscr{H}^{s}\left(C\cap\left(C+x\right)\right)=L_{x}\ge\liminf_{j\to\infty}\left\{ 3^{j}\right\} =\infty$.
\end{example}
Theorem \ref{thm:m^-bL<H<L} requires that $t$ does not admit finite
$n$-ary representation and that $\sigma_{t}\left(k\right)=\pm1$
for all $k$. The infinite representation requirement allows us to
ignore the potentially empty and empty cases by Lemma \ref{Sec-3-Lem:empty-can-be-ignored}.
The requirement that $\sigma_{t}\left(k\right)=\pm1$ for all $k$
allows us to not only count the total number of intervals and potential
intervals of $C_{k}$ using the function $\mu_{t}\left(k\right)$,
but also guarantees that all intervals and potential intervals contain
points in $C\cap\left(C+t\right)$. 

Note that $L_{t}$ is calculated by counting all interval and potential
interval cases at each step $k$. The following example demonstrates
when potential interval cases do not lead to points in $C\cap\left(C+t\right)$,
thus showing the necessity of Lemma \ref{Sec-4-Lem:Interval_Count}
and Lemma \ref{Sec-4-lem:Nothing_Goes_Away_1} to the calculations
in Theorem \ref{thm:m^-bL<H<L}:
\begin{example}
\label{ex: NGA-Required} Let $D=\left\{ 0,2,4,7,10,\cdots,4+3r\right\} $
for some integer $r>2$ and $n>4+3\left(r+1\right)$ so that $C=C_{n,D}$
is not sparse. Let $t:=0._{n}\overline{2}$ so that $\sigma_{t}\left(k\right)=i$
for all $k$. For each $k$, $C_{k}\cap\left(C_{k}+\left\lfloor t\right\rfloor _{k}\right)$
contains $2^{k}$ interval cases and $r\cdot2^{k-1}$ potential interval
cases, however the potential interval cases never contain points in
$C\cap\left(C+t\right)$ since $2$ is neither in $n-\Delta$ nor
$n-\Delta-1$. By calculation,
\begin{align*}
\beta_{t} & =\liminf_{k\to\infty}\left\{ \frac{\log_{m}\left(2+r\right)+\log_{m}\left(2^{k-1}\right)}{k}\right\} =\log_{m}\left(2\right)\\
L_{t} & =\liminf_{k\to\infty}\left\{ \left(2+r\right)\cdot2^{k-1}\cdot m^{-\beta_{t}k}\right\} =\frac{2+r}{2}.
\end{align*}
Thus, $m^{-\beta_{t}}=\frac{1}{2}$ and $\left[m^{-\beta_{t}}L_{t},L_{t}\right]=\left[\frac{2+r}{4},\frac{2+r}{2}\right]$
by the same method as Theorem \ref{thm:m^-bL<H<L}. We will show that
the Hausdorff measure at most $1<\frac{2+r}{4}$:

Since potential interval cases never contain points in $C\cap\left(C+t\right)$,
we can instead perform the same calculations using only the interval
cases as a cover of $C\cap\left(C+t\right)$. Thus, $\beta_{t}=\log_{m}\left(2\right)$
and $s:=\log_{n}\left(2\right)$ so that $\mathscr{H}^{s}\left(C\cap\left(C+t\right)\right)\le\liminf_{k\to\infty}\left\{ 2^{k}\cdot m^{-\beta_{t}k}\right\} =1$.
Thus, the calculation of $L_{t}$ gives an incorrect result even though
$\beta_{t}$ is calculated properly. 
\end{example}

\section{Open Questions}

It is known that integral self-affine sets must have rational Lebesgue
measure \cite{BoKr11} so, perhaps, the range of $t\mapsto\mathscr{H}^{s}\left(C\cap\left(C+t\right)\right)$
is not all of the interval $[0,\infty).$ See also Example \ref{ex:L_t is nowhere dense.}. 

It is likely that our methods provided an estimate of the Hausdorff
measure of $C_{n,D_{1}}\cap\left(C_{n,D_{2}}+t\right),$ simply by
replacing the sparcity condition by the assumption that $\left|\delta-\delta'\right|\geq2$
for all $\delta\neq\delta'$ in $D_{1}-D_{2}.$

\end{document}